\documentclass[12pt,epsfig,amsfonts]{article}
\usepackage{amsmath,amssymb,epsfig, amsthm, geometry}
\newtheorem{proposition}{Proposition}[section]

\newtheorem{theorem}{Theorem}
\newtheorem{lemma}[proposition]{Lemma}

\newcommand{\ve}{\varepsilon}
\newcommand{\Si}{\mathcal{S}}

\newcommand{\C}{\mathcal{C}}
\newcommand{\E}{\mathcal{E}}

\newcommand{\R}{\mathbb{R}}
\newcommand{\tf}{\tilde f}

\newcommand{\G}{\mathcal{G}}
\newcommand{\I}{\mathcal{I}}

\newcommand{\musrb}{\mu_{\mbox{\tiny SRB}}}
\newcommand{\tmusrb}{\tilde{\mu}_{\mbox{\tiny SRB}}}
\newcommand{\Gsrb}{\G_{\mbox{\tiny SRB}}}
\newcommand{\Lp}{\mathcal{L}}
\newcommand{\pa}{\mathcal{P}}

\newcommand{\bDelta}{\overline{\Delta}}
\newcommand{\bpi}{\overline{\pi}}
\newcommand{\barF}{\overline{F}}
\newcommand{\dlj}{\Delta_{\ell,j}}
\newcommand{\bm}{\overline{m}}
\newcommand{\lip}{\mbox{\tiny Lip}}
\newcommand{\bH}{\overline{H}}

\newcommand{\B}{\mathcal{B}}
\newcommand{\bLp}{\overline{\Lp}}
\newcommand{\ra}{\mathfrak{r}}
\newcommand{\tB}{\tilde{\B}}

\newcommand{\bpsi}{\overline{\psi}}
\newcommand{\tmu}{\tilde{\mu}}
\newcommand{\tnu}{\tilde{\nu}}
\newcommand{\F}{\mathring{F}}
\newcommand{\f}{\mathring{f}}
\newcommand{\M}{\mathcal{M}}
\newcommand{\bnu}{\overline{\nu}}

\newcommand{\bareta}{\overline{\eta}}
\newcommand{\vf}{\varphi}
\newcommand{\bvf}{\overline{\varphi}}
\newcommand{\tvf}{\tilde{\varphi}}

\newcommand{\Z}{\mathcal{Z}}

\newcommand{\N}{\mathbb{N}}

\begin{document}

\title{Entropy, Lyapunov Exponents and Escape Rates\\
in Open Systems}

\author{Mark Demers\thanks{Department of Mathematics and Computer Science,
Fairfield University, Fairfield, USA.  Email: mdemers@fairfield.edu.  This research
is partially supported by NSF grant DMS-0801139.}
\and Paul Wright\thanks{Department of Mathematics, University of Maryland, College Park, USA.  Email: paulrite@math.umd.edu.}
\and Lai-Sang Young\thanks{Courant Institute of Mathematical Sciences, New York University, New York, USA.  Email: lsy@cims.nyu.edu.}
}

\maketitle

\begin{abstract}
We study the relation between escape rates and pressure in general dynamical systems with holes, where pressure is defined to be the difference between entropy and the sum of positive Lyapunov exponents.
Central to the discussion is the formulation of a
class of invariant measures supported on the survivor set over which we take the supremum to measure the pressure.
Upper bounds for escape rates are proved for general
diffeomorphisms of manifolds, possibly with singularities, for arbitrary holes
and natural initial distributions including Lebesgue and SRB measures.
Lower bounds do not hold in such generality, but for systems admitting
Markov tower extensions with spectral gaps, we prove the
equality of the escape rate with the absolute value of the pressure and the existence of
an invariant measure realizing the escape rate,
i.e.\ we prove a full variational principle.
As an application of our results, we prove a variational principle for the billiard map associated with a planar Lorentz gas of finite horizon with holes.
\end{abstract}


\section{Introduction}

This paper is about {\it leaky dynamical systems} or dynamical systems with
{\it holes}. A generic setup consists of a triple $(f, M; H)$ where
$M$ is the phase space of a map or flow denoted by $f$, and
$H \subset M$ is an open set. We refer to $(f,M)$ as a {\it closed system} and
$H$ as the hole through which
mass is allowed to escape from the system.
More precisely, we follow trajectories in $M$ until they enter $H$.  Once a
point enters $H$, it leaves the system forever, i.e. we stop considering it.

Holes can be large or small. Small holes are often used to model small
(unintended) leaks in physical systems; proximity of normalized surviving
distributions to the physical measure of the closed system is a form of stability. More generally, the study of
$(f, M; H)$ can be viewed as the study of dynamics on {\it non-invariant domains}. As an example of why such studies
are relevant, consider the following. It is well known that attractors
are important because they capture the large-time behavior of dynamical systems,
but invariant sets  that are not attracting can
substantially impact the qualitative behavior of a system as well: Let
$\Lambda \subset M$ be such a set, and $U \subset M$ a neighborhood of
$\Lambda$. Then we may regard $H = M \setminus \overline U$ as the hole.
Slow escape rates from such holes are known to impact the speed of correlation
decay of the closed system.

Escape dynamics have been studied by many authors. We refer the reader to the
part-review article \cite{demers young}, which contains many references, and will mention
explicitly works that are closer to the present paper as we go along.
Most previous works have focused on specific systems, such as Anosov diffeomorphisms,
interval and billiard maps. In this paper, we seek a general understanding
 for as large a class of dynamical systems as we can.
Specifically, we seek to relate escape rate to a dynamical invariant
called  {\it pressure}, which
roughly speaking measures the discrepancy between metric entropy
and sum of positive Lyapunov exponents. We now proceed to a discussion
of what this paper is about.

\bigskip
\noindent {\bf Setting and questions}

\smallskip

We begin with the simpler setting of a compact Riemannian manifold $M$
without boundary and a diffeomorphism $f$  which is at least $C^{1+\epsilon}$
for some $\epsilon>0$. In order to include applications to systems such as
billiards, which are very important examples of dynamical systems of
physical origin, we also allow $M$ to be the union of a (possibly open)
Riemannian manifold
and a singularity set $\Si$, and $f$ to be piecewise smooth.
Precise conditions on $\Si$ and the behavior of $f$ near it will be
introduced in Section 2. Riemannian measure on $M$ (or $M \setminus \Si$)
is denoted by $\mu$ throughout. Unless otherwise stated, the hole
$H$ is an arbitrary open set in $M$.

Let $m$ be a reference measure on $M$. We think of $m$ as the initial
distribution of mass in the phase space before any escape takes place, and
take the view that initial distributions related to $\mu$ are of particular physical interest. Notice that $m$ need not be $f$-invariant. Indeed one can
interpret the situation as follows: The escape of mass can begin before
or after the closed system $f: M \circlearrowleft$ reaches a steady state.
In the first case,
$m$ is usually not invariant, and we assume it has a density with respect to $\mu$.
In the second case, we take $m$ to be an SRB measure, which may be singular
with respect to $\mu$.

A basic quantity of interest is the
{\it escape rate}, defined to be $-\rho(m)$ where
\begin{equation}
\label{eq:escape def}
\rho(m) \ = \ \lim_{n \to \infty} \frac{1}{n} \log m(M^n)\
\end{equation}
when the limit exists.
Here $M^n = \cap_{i=0}^n f^{-i}(M \setminus H)$ is the set of points which has
not escaped by time $n$. In general, the limit in (\ref{eq:escape def})
may not exist, and we write
$\underline \rho$ and $\overline \rho$
for the $\liminf_{n \to \infty}$ and $\limsup_{n \to \infty}$ of the quantity
on the right hand side. Notice that while $\rho(m)$ depends on $m$,
all initial distributions uniformly equivalent to $m$ have the same escape
rate, i.e.  if $\vf$ is a function with
$\frac{1}{c} \le \vf \le c$ for some $c>0$, then $\rho(\vf m) = \rho(m)$,
and the same is true for $\underline \rho$ and $\overline \rho$.

For an $f$-invariant Borel probability measure $\nu$ on $M$, the {\it pressure}
of $\nu$, denoted $P_\nu$, is defined to be
$$P_\nu \ = \ h_\nu(f) - \int \lambda^+ d\nu
$$
where $h_\nu(f)$ is the metric entropy of $(f, \nu)$ and $\lambda^+$ is the
the sum of the positive Lyapunov exponents counted with multiplicity.
We will write
$\pa_{\cal G} = \sup_{\nu \in {\cal G}}  P_\nu$
where $\cal G$ is a collection of invariant measures.

Given an open system $(f,M; H)$, we define the {\it survivor set}
to be the $f$-invariant set
$\Omega := \cap_{n \in {\mathbb Z}} f^n(M \setminus H)$.\footnote{If
$f$ is not invertible, we take $n \leq 0$ in the definition of $\Omega$.}
Let $\I = \I(\Omega)$ denote
the set of $f$-invariant
Borel probability measures supported on $\Omega$,
and let $\E \subset \I$ be the subset of $\I$ consisting of ergodic measures.
Assuming $\rho(m)$ is well defined, we say $\rho(m)$ satisfies a
{\it variational principle} if
$$
\rho(m) = \pa_{\cal G} \qquad {\rm for  \ a \ suitable  \ class \ of \ measures}
 \ \cal G \subset \I \ .
$$
Of interest also is whether
 the supremum in $ \pa_{\cal G}$ is attained, i.e. if there is a measure
$\nu \in {\cal G}$ for which $P_\nu =  \pa_{\cal G}$. Obviously, one can also ask
if $\rho(m)=P_\nu$ for some $\nu$
without mentioning any variational principles.

The ideas in the last paragraph were suggested by a number of previously
known results some of which are recalled below, but let us first summarize the
questions to be addressed.

\bigskip
This paper seeks to address for as large a class of dynamical systems
as possible the following three questions for natural initial distributions $m$:

{\it
\begin{itemize} \vspace{-6 pt}
\item[{\bf Q1}] {\bf (Escape rate)} \
Is the escape rate $- \rho(m)$ well defined?
\vspace{-6 pt}
\item[{\bf Q2}] {\bf (Formula for escape rate)}
Is $\rho(m) = h_\nu(f) - \int \lambda^+ d\nu$ for some $\nu \in \I$?

\vspace{-3 pt}
The same question can be posed for  $\underline \rho(m)$ and
$\overline \rho(m)$.
\vspace{-6 pt}
\item[{\bf Q3}] {\bf (Variational principle)} \ Does $\rho(m)$ satisfy a variational principle?
\end{itemize}
}
\vspace{-3 pt}
Partial answers are given for very general dynamical systems, and complete
answers for a more restricted class which includes many known examples.
 A concrete application to the leaky periodic Lorentz gas is mentioned explicitly.

\bigskip

\noindent {\bf Earlier works}

\begin{theorem} {\rm \cite{bowen}} Consider a $C^{1+\epsilon}$ Axiom A
diffeomorphism $f:M\circlearrowleft$  of a compact Riemannian manifold $M$.
Let $\Lambda \subset M$ be a basic set, and let $\I=\I(\Lambda)$.
Then $\pa_{\I} \le 0$,
and $\pa_\I =0$ if and only if $\Lambda$ is an attractor.
\end{theorem}

This is the first result that systematically
relates the escape of mass to pressure: In the case where $\Lambda$
is an Axiom A attractor,
no mass can escape from a neighborhood of $\Lambda$,
and $\pa_{\I} = 0$; for non-attracting basic sets such as horseshoes,
mass escapes at exponential rates and $\pa_{\I} < 0$.
The number $\pa_{\I}$ has been shown to be equal to the
{\it topological pressure} of $f$ with respect to the potential
$-\log |\det(Df^u)|$ on $\Lambda$; see \cite{bowen} or \cite{walters} for more detail.

The next result gives
conditions under which the numerical value of $\pa_{\I}$ is explicitly related
to the rate of escape.

\begin{theorem} {\rm \cite[Theorem 4]{young large d}}\footnote{This result follows
from the large deviation results in Theorem 1 (not Theorem 2) of \cite{young large d}.
Take $\vf \equiv 1$
on a closed set $K$ and $<1$ on $M \setminus K$ where
$\Omega \subset {\rm int}(K) \subset K \subset M \setminus H$,
and $\xi \approx -\log |\det(Df|_{E^u})|$ on $\Omega$.}
Let $f:M\circlearrowleft$ be a $C^{1+\epsilon}$
diffeomorphism of a compact Riemannian manifold $M$, and let
$H \subset M$ be an open set. We assume

\smallskip
(i) $\Omega$ is compact with $d(\Omega, \partial H)>0$, and

(ii) $f|_\Omega$ is uniformly hyperbolic.

\smallskip
\noindent Then $\rho(\mu)$ is well defined and equals $\pa_{\I}$.
\end{theorem}

In both of the settings above, $\pa_\I = \pa_\E$, and $\pa_{\I}=P_\nu$
for some $\nu \in \I$. (The latter follows from the
continuity of $x \mapsto \log |\det(Df|_{E^u})|$ and upper semicontinuity of
$\nu \mapsto h_\nu(f)$; see [B]). Thus for  uniformly hyperbolic survivor
sets $\Omega$ with $d(\Omega, H)>0$, {\bf Q1}--{\bf Q3} have all been
answered in the affirmative.

Several works went beyond Theorem 2 to give positive answers to
{\bf Q1} and {\bf Q2} in a number of situations, including Anosov diffeomorphisms
with Markov or small holes (with no requirement on $\Omega \cap \partial H$)
\cite{chernov mark1, chernov mark2, chernov mt},
uniformly expanding maps admitting Markov partitions \cite{collet},
piecewise expanding maps, and Collet-Eckmann maps of the interval with singularities \cite{bdm}.  {\bf Q3} was partially addressed in \cite{chernov mark1, collet, bdm}:
a variational principle was proved for an associated dynamical system, namely
the symbolic dynamics of the original map (but not for the map itself).

\section{Statement of Results}

Three sets of results are stated:
\begin{itemize} \vspace{-6 pt}
\item[--] Sect.~2.1 contains partial answers to
{\bf Q3}: lower bounds for $\underline \rho(m)$ are proved for very
general dynamical systems; no results on upper bounds
are reported.
\vspace{-6 pt}
\item[--] Sects.~2.2 and 2.3 provide complete
answers to {\bf Q1}--{\bf Q3} for systems admitting
Markov tower extensions with some additional conditions.
\vspace{-6 pt}
\item[--] These results are applied to the periodic Lorentz
gas with small holes (Theorem F).
\end{itemize}


\subsection{Lower bounds on $\underline \rho(m)$ for general dynamical systems}
\label{upper bound}

Our results in this subsection will assert, in essence, that

\medskip
\centerline{\it for very general dynamical systems,
$\underline \rho(m) \ge \pa_\G$ for
reasonable choices of $\G$.}

\medskip
\noindent
Since $\pa_\G$ decreases with $\G$, this inequality is not meaningful
for $\G$ too small. Thus the selection of a suitable $\G$ is an important part of
the consideration.
We start with $\E$, the set of ergodic invariant measures supported on the survivor set $\Omega$. To obtain $\G$, restrictions will be placed on $\E$ on account of

\smallskip
I. the hole $H$,

II. the initial distribution $m$, and

III. singularities of the map $f$, if present.

\smallskip
\noindent We discuss these 3 types of restrictions separately.
The conditions we impose  are admittedly motivated by our proofs,
but the fact that they lead to
a full variational principle for a large class of dynamical systems (see Sect. 2.3)
suggests that these choices of $\G$ are reasonable.

\bigskip
\noindent {\it Remark.}
One should keep in mind that the escape rate is defined by $-\rho(m)$ when interpreting
the inequality $\underline \rho(m) \ge \pa_\G$.  Thus a lower bound for $\underline \rho(m)$
provides an upper bound of $|\pa_\G|$ for the escape rate.

\bigskip
In Paragraphs I and II below, $f:M\circlearrowleft$ is a $C^{1+\epsilon}$ diffeomorphism; systems with singularities are discussed in Paragraph III.
Throughout the paper, $B(x,r)$ denotes the ball of radius $r$ in $M$
centered at $x \in M$, and
$N_\ve(\cdot)$ denotes the $\ve$-neighborhood of a set in $M$.

\bigskip
\noindent {\bf I. Restrictions on $\G$ due to the hole $H$}

\medskip
\noindent
The following
definition gives a sense of which $\nu \in \E$ we think impact the escape rate. Define
\begin{eqnarray*}
\G_H & = & \{ \nu \in \E \mid \mbox{The following holds for  $\nu$-a.e. $x$: given any
$\gamma>0$,} \\
& & \mbox{\qquad $\exists r = r(x, \gamma)>0$ such that
$B(f^ix, re^{-\gamma i}) \subset M \setminus H$ for all $i \ge 0$} \}\ .
\end{eqnarray*}
Notice that if $\nu \in \E$ has the property that for some $C, \alpha >0$,
$\nu(N_\ve(\partial H)) \leq C \ve^\alpha$ for all $\ve >0$, then
$\nu$ is in $\G_H$ (see Sect. 4.2, Paragraph 4).

The definition of $\G_H$ can be relaxed in many ways; in particular,
it is not necessary for the entire ball $B(f^ix, re^{-\gamma i})$ to be in
$M \setminus H$. We mention one formulation,
leaving the reader to contemplate others:
Given $x \in M$, let $W^s_\ve(x)$ denote the local stable manifold of $x$ of
radius $\ve$.
We call an open set $O$ a {\em $W^s$-neighborhood of $x$} if
$O \cap W^s_\ve(x) \ne \emptyset$ for every $\ve>0$.
All of our results remain valid if

\bigskip
\noindent
\parbox{.1 \textwidth}
{\bf (O)}
\parbox[t]{.8 \textwidth}
{\it  in the definition of $\G_H$, $B(f^ix, re^{-\gamma i})$ is replaced by
$f^i(O) \cap B(f^ix, re^{-\gamma i})$
where $O$ is a $W^s$-neighborhood of $x$.}

\bigskip
\noindent
{\bf II. Restrictions on $\G$ due to the initial distribution $m$}

\medskip
Two types of initial distributions are considered.

\bigskip
\noindent {\bf (A) Initial distributions with densities, possibly localized}

\smallskip
Let $m=\mu_\vf = \vf \mu$ where $\vf \ge 0$
is in $L^1(\mu)$. For such an initial distribution, we consider
$$
\G_\vf = \{\nu \in \E : \exists \, c_\nu>0 \ {\rm and \ an \ open \ set} \ Z {\rm \ such \
that} \ \nu(Z) >0 \mbox{ and } \vf|_Z \geq c_\nu\}\ .
$$

\smallskip
\noindent {\bf Theorem A.} {\it Let $(f,M;H)$ be as above. Then

(i) $\underline \rho(\mu) \ge \pa_{\G_H}$;

(ii) more generally, $\underline \rho(\mu_\vf)
\ge \pa_{\G_H \cap \G_\vf}$.}

\bigskip
\noindent {\it Remark.} Clearly, $\G_\vf = \E$ if $\vf \ge c$ for some $c>0$;
thus (ii) reduces to (i). Here we permit $\vf$ to vanish on parts of $M$ provided
it is measurable with
ess inf$(\vf)>0$ on an open set of $M$. We do not claim that the restrictions
imposed on $\G_\vf$ are necessary, but  if the support of $\vf$ is localized in the phase space,
invariant measures supported elsewhere are clearly irrelevant since they cannot
be ``seen" by the initial distribution $\mu_\vf$.

\bigskip
\noindent
{\bf (B) SRB measures as initial distributions}

\smallskip
In (A), $m=\mu_\vf$ is not
necessarily an invariant measure.
If, however, a steady state is reached before the leak begins, then
it would be natural to take $m$ to be an SRB measure $\musrb$, as we now do.
For simplicity, we assume $\musrb$ has no zero Lyapunov exponents.

The challenge here is to identify a class of invariant measures $\Gsrb$ that
can be ``seen" by the SRB measure $\musrb$, which is often singular.
We call $\Pi \subset M$ a {\em $\musrb$-hyperbolic product set}
if the following hold.

\begin{enumerate}
  \item[(W.1)]  $\Pi = (\cup \Gamma^u) \cap (\cup \Gamma^s)$ where
  $\Gamma^{u} = \{ \omega \}$ and $\Gamma^{s} = \{ \omega' \}$
  are two sets of relatively open local unstable and stable manifolds such that
  each $\omega \in \Gamma^u$ intersects
  every $\omega' \in \Gamma^s$ in precisely one point.  In addition, there exist constants $C>0$,
  $\lambda < 1$ such that
  \[  \mbox{diam}(T^{-n}\omega) \le C \lambda^n \; \;  \forall \omega \in \Gamma^u
  \; \; \; \; \mbox{and} \; \; \; \;
  \mbox{diam}(T^n\omega') \le C \lambda^n \; \; \forall \omega' \in \Gamma^s ,
  \]
    where diam$( \cdot )$ denotes the diameter of the unstable or stable manifold.
  \vspace{-6 pt}
    \item[(W.2)]  $\musrb|_{\Pi}(A)>0$ for every relatively open
    $A \subset \Pi$.
    \vspace{-6 pt}
    \item[(W.3)]  There exists a constant $c_\Gamma >0$ such that for $\musrb$-a.e.\
    $\omega \in \Gamma^u$, the conditional probability of
  $\musrb$ on $\omega$ has density $\psi_\omega \geq c_\Gamma$.
\end{enumerate}
We remark that (W.3) is a general property of SRB measures [LY]; we have
listed it separately only for emphasis. Define
$$
\Gsrb \ = \ \{ \nu \in \E \mid \nu(\Pi)>0 \ \mbox{for a
$\musrb$-hyperbolic product set $\Pi$} \}  .
$$

\smallskip
\noindent {\bf Theorem B.} {\it Under the conditions above,
$\underline \rho(\musrb) \ge \pa_{\G_H  \cap \Gsrb}$.}

\bigskip
\noindent {\it Remark 1.} Observe that if $f$ has an Axiom A attractor $\Lambda$
and $\musrb$ is the
SRB measure on the attractor, then $\Gsrb$ imposes no restriction whatsoever
on $\nu \in \E$, i.e. $\Gsrb = \E$.

\medskip
\noindent {\it Remark 2.} In the case where the pushforward of
Lebesgue measure $\mu$ tends to $\musrb$, one might be tempted to
conclude
that $\rho(\mu)=\rho(\musrb)$. This is {\it not necessarily true}, and the
reason is as follows: Suppose
$f$ has a Lebesgue measure zero invariant set $\Lambda$
(such as a horseshoe) away from the support of the SRB measure.
The rate at which points escape from a neighborhood of $\Lambda$ will
be reflected in $\rho(\mu)$ but not in $\rho(\musrb)$; this can easily lead to
$\rho(\mu) > \rho(\musrb)$.

\bigskip \noindent
{\bf III. Restrictions on $\G$ due to the singularities of the map $f$}

\medskip
We state here a version of our results
that can be applied to planar billiards; see Theorem F below. Following \cite{katok}, we let
$U$ be an open smooth (at least $C^4$) finite dimensional
Riemannian manifold, and assume that $M = \overline U$ is a compact
metric space of finite capacity,\footnote{This means there is some $d<\infty$ such that
$
\limsup_{r \to 0} \frac{\log C(r)}{-\log r} = d
$
where $C(r)$ is the minimum cardinality of a covering of $M$ by open balls of
radius $r$. For billiards with corners the set $\overline U$ is technically
not a manifold with boundary but a union of such glued together
along some boundaries.} where $\overline U$ denotes the closure of $U$.

Let $\iota(x,U)$ be the radius of injectivity of the
exponential map exp$_x: T_xU \to U$.  We assume that there exist constants
$s, c_0, \varsigma>0$ such that for each $x,y \in U$ such that $d(x,y) < \iota(x,U)$
and $w = \mbox{exp}_x^{-1}(y)$, we have
\begin{equation}
\label{eq:exp}
\iota(x,U) \geq \min \{ s, d(x, M\setminus U)^\varsigma \} ,  \; \; \;
\| D (\mbox{exp}_x)(w) \| \leq c_0, \; \; \mbox{and} \; \; \| D(\mbox{exp}_x^{-1})(y) \| \leq c_0 .
\end{equation}

Let $V$ be an open subset of $U$ and let $f:V \to U$ be a mapping which is
a $C^2$ diffeomorphism of $V$ onto its image.
Let $\Si = M\setminus V$.  We think of $\Si$ as the singularity set of $f$.
We assume that
there exist constants $C_1, a>0$ such that for all $x \in V$,
\begin{equation}
\label{eq:d1 blowup}
\| Df_x \| \leq C_1 d(x,\Si)^{-a} \qquad \mbox{and} \qquad
\| Df^{-1}_x \| \leq C_1 d(x,f\Si)^{-a}.
\end{equation}
Let
$\hat{f}_x = \mbox{exp}_{fx}^{-1} \circ f \circ \mbox{exp}_x$
denote the induced map on $T_xV$ wherever it is defined.
We assume that there exists $b >0$ such that if $x \in V$, $v \in T_xV$ and
$\hat{f}_x(v)$ is well-defined, then
\begin{equation}
\label{eq:d2 blowup}
\| D^2 \hat{f}_x(v) \| \leq C_1 d(\mbox{exp}_x(v), \Si)^{-b} .
\end{equation}
Notice that for billiards with finite horizon, $a = 1$ and $b=3$
(see \cite{katok, chernov book}).  In what follows, we will assume without loss of
generality that $b \geq \varsigma \geq 1$.

Into such a system we introduce a hole $H \subset M$. With regard to the
choice of $\G$, in addition to the considerations above, we must also
restrict to invariant
measures that respect the singularities (see [KS]). Define
$$
\G_{\Si} = \{\nu \in \E \mid \exists C, \alpha>0 \
{\rm such \ that} \ \forall \ve>0, \nu(N_\ve(\Si)) \leq C\varepsilon^\alpha\}\ .
$$

\smallskip
\noindent {\bf Theorem C.} {\it  Let $(f,M;H)$ be as above. Then
\begin{enumerate} \vspace{-6 pt}
\item[(i)] for $\vf \in L^1(\mu)$, $\underline \rho(\mu_\vf) \ge \pa_{\G_H \cap \G_{\Si}
\cap \G_\vf}$, and
\vspace{-6 pt}
\item[(ii)] if $f$ has an SRB measure $\musrb$ with no zero Lyapunov exponents, then \\
$\underline \rho(\musrb) \ge \pa_{\G_H \cap \G_{\Si}
\cap \Gsrb}\ .$
\end{enumerate}
}

\bigskip
We finish with the following.

\medskip
\noindent {\bf Remarks on upper bounds and the attractor case:}
No general results are known for upper bounds on $\rho(m)$, not even for
$m=\mu$. Consider the special case where
$\Omega \subset M$ is an attractor. Assume there is a neighborhood
$O$ of $\Omega$ such that $f(\overline O) \subset O$ and
$\Omega = \cap_{n \ge 0} f^n(O)$.
Let $H = M \setminus \overline O$ and $m=\mu$,
so that $\rho(m)=0$ by definition.
Since $h_\nu(f) \le \lambda_\nu^+$ for all $\nu \in \E$ \cite{ruelle},
showing that $\rho(m) \le \pa_\G$ in this case is equivalent to proving
$\pa_\E =0$. The latter is known to be false in general, an example being the
Figure 8 attractor (see Fig.~1), so one must rephrase the question to include
some notion of ``typicality". Still, $P_\nu=0$ means either
$\lambda^+_\nu=0$ or $\nu$ is an SRB measure \cite{ledrappier young},
and whether attractors with nonuniform expansion admit SRB
measures is well known to be a very difficult question;
see e.g.\ \cite{young srb}.
Since any result on upper bounds for $\overline \rho(m)$ must include
this attractor case, we conclude that in complete generality
the question for upper bounds for $\overline \rho(m)$ (and lower bounds for escape rates)
is intractable at the present time.

\medskip
We will, however, identify a large class of dynamical
systems for which $\underline \rho(m)= \overline \rho(m) = \pa_\G$
for some $\G$. This is the content of Sects. 2.2 and 2.3.

\bigskip
\begin{center}
\includegraphics{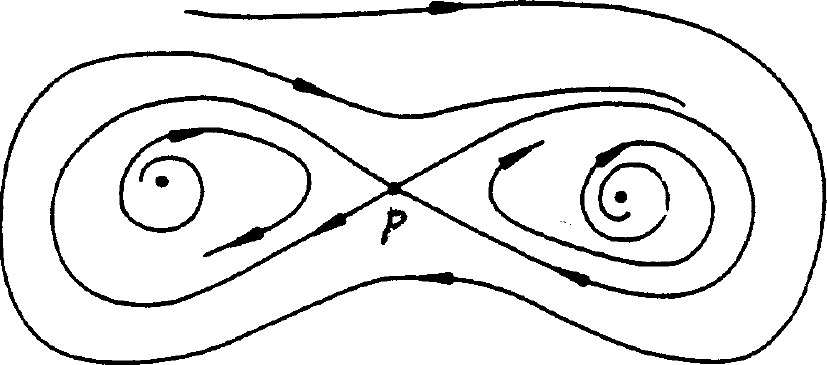}
\end{center}

\centerline{Fig.\ 1. Figure 8 attractor. The only invariant measure is $\delta_p$
where $p$  is the saddle point.}


\subsection{Escape rate formula}
\label{tower results}

In this section, we assert for a class of dynamical systems
the existence of $\hat \nu \in \E$ the pressure of which is equal to $\rho(m)$, thereby answering {\bf Q1} and
{\bf Q2} in the affirmative.

Let $f:M\circlearrowleft$ be
a $C^{1+\ve}$ diffeomorphism or a piecewise smooth diffeomorphism
as in the setting of Theorem C, and fix a hole $H \subset M$.
We assume

\medskip
{\bf (A.1)} $(f,M)$ has a {\it Markov tower extension} $(F, \Delta)$;

\medskip
{\bf (A.2)} $(F, \Delta)$ has an {\it exponential tail};

\medskip
{\bf (A.3)} $(F, \Delta)$ {\it respects the hole} $H$;

\medskip
{\bf (A.4)} the transfer operator on the ``tower with holes" has a {\it spectral gap}.

\medskip
\noindent  While {\bf (A.1)} and {\bf (A.2)} are by now quite standard,
and {\bf (A.3)} and {\bf (A.4)} have also appeared elsewhere,
it will take a few pages to make precise this entire formal setting;
we postpone that to Sect.~\ref{tower review}.  Let $\musrb$ denote
the (unique) ergodic SRB measure on $\pi(\Delta)$ where
$\pi: \Delta \to M$
is the projection, and let $\ra<1$ be
the leading eigenvalue of the transfer operator on the tower with holes.

We will use the following notation:
Let $m = m^{(0)}$ denote a probability measure on $M$.
For $n \ge 1$, let $m^{(n)}$ denote the normalized surviving distribution
at time $n$, i.e.\ $m^{(n)} = f^n_*(m|_{M^n})/m(M^n)$, assuming $m(M^n)>0$.
We call a measure $m$ conditionally invariant with eigenvalue
$t$ if $m$ is supported on $M\setminus H$ and $f_*(m|_{M^1}) = t \, m$.

\bigskip
\noindent {\bf Theorem D.} {\it Assume $(f,M;H)$ satisfies {\bf (A.1)}--
{\bf (A.4)}. Then
\vspace{-6 pt}
\begin{enumerate}
\item[(a)]  $\rho(\musrb)$ is well defined and equals $\log \ra$;
\vspace{-6 pt}
\item[(b)] $\musrb^{(n)}$ converges weakly to a conditionally invariant measure
$\mu_*$ with eigenvalue $\ra$;
\item[(c)]  there exists $\hat \nu \in \G_H \cap \G_\Si$ such that
$$\rho(\musrb) = P_{\hat \nu} := h_{\hat \nu}(f) - \lambda^+_{\hat \nu}\ ;$$
\item[(d)] $\hat \nu$ is defined by
$$
\hat \nu(\vf) = \lim_{n \to \infty} \ra^{-n} \int_{M^n} \vf \, d\mu_*  \qquad \mbox{ {\it for  all
continuous}} \ \vf\ .
$$
In addition, $\hat \nu$ enjoys exponential decay of correlations on H\"older observables.
\vspace{-6 pt}
  \end{enumerate}
  }

\bigskip
Our construction of $\hat \nu$ generalizes that in \cite{collet, chernov mark1},
which assume the maps in question admit finite Markov partitions. See
\cite{bdm} for the first
generalization in this direction regarding pressure for one-dimensional maps with holes.
Parts (a) and (b) of Theorem D are also known
for the periodic Lorentz gas \cite{dwy}. We assert here that these results hold generally for {\it any} dynamical system
admitting a tower with the stated conditions.


\subsection{A full variational principle}

Combining the results of the previous two sections, we are able to state a full variational principle (answering {\bf Q1}--{\bf Q3} in Section 1)
for maps admitting towers with a spectral gap as described in Sect.~\ref{tower results}.
Let $\Lambda \subset M$ be the reference
hyperbolic product set which forms the base of the tower $\Delta$.

\bigskip
\noindent {\bf Theorem E.}  {\it Assume $(f,M;H)$ satisfies {\bf (A.1)}--
{\bf (A.4)},  and let $\hat \nu$ be as in Theorem D.
\begin{enumerate} \vspace{-6 pt}
\item[(a)]  If $\musrb = \vf \mu$ where $\vf \geq \delta >0$ on a neighborhood
of $\Lambda$,  then
$\hat \nu \in \G_H \cap \G_{\Si} \cap \G_\vf$ and
\[
\rho(\musrb) = P_{\hat \nu} = \pa_{\G_H \cap \G_{\Si} \cap \G_\vf} .
\]
\item[(b)] If $\Lambda$ is contained
in a $\musrb$-hyperbolic product set, then $\hat \nu \in
\G_H \cap \G_{\Si} \cap \Gsrb$ and
\[
\rho(\musrb) = P_{\hat \nu} = \pa_{\G_H \cap \G_{\Si} \cap \Gsrb} .
\]
  \end{enumerate}
}

\medskip
To our knowledge the condition in part (b) of Theorem E can be arranged in
all known tower constructions.

\bigskip
\noindent {\bf Remark on results for tower maps.} We will, as an intermediate step to proving
Theorems D and E, prove the corresponding results for tower maps
with Markov holes. These results are stated as Theorems 4 and 5 in
Sect. 5.2.

\bigskip
\noindent
{\bf An illustrative example: The 2D periodic Lorentz gas}

\smallskip
We conclude this section by stating an application of our results to a concrete example.
The setting here is as in \cite{dwy}:
Let $f:M\circlearrowleft$ be a billiard map
associated with a two dimensional periodic Lorentz gas with finite horizon
whose scatterers are bounded by $C^3 $ curves with strictly positive curvature.
The holes we introduce into $M$ are derived from two types of holes in the billiard table $X$.
We say $\sigma \subset X$ is a hole of Type I if $\sigma$ is an open segment of
an arc in the boundary of one of the scatterers in $X$.  We say $\sigma$ is a hole
of Type II if it is an open convex set in $X$ whose closure is disjoint from any of
the scatterers.  The hole $\sigma \subset X$ induces a hole $H_\sigma \subset M$
which we also call a hole of Type I or Type II. See \cite{dwy} for more general
holes and details on the geometry they induce in $M$.

\bigskip
\noindent {\bf Theorem F.} {\it Let $f$ be the billiard map in the last
paragraph. Let
$H_\sigma$ be a hole of Type I or Type II, and assume it is small enough in the sense of \cite{dwy}.  Then

\smallskip
(a) $\rho(\musrb)=\pa_{\G_H \cap \G_\Si}$;

\smallskip
(b) there exists $\hat \nu \in \G_H \cap \G_\Si$ such that $P_{\hat \nu} = \rho(\musrb)$.
}

\bigskip
Theorem F is an immediate consequence of Theorems D and E together
with \cite{dwy}:  In \cite{dwy}, towers with exponential tails respecting arbitrary holes
of Types I and II  are
constructed, and for small enough holes  the spectral gap property is guaranteed.
Thus the conditions for Theorem~D are satisfied; however, \cite{dwy} does not
address variational principles or pressure so that Theorem D, parts (c) and (d), as well
as Theorem F are new results for this class of billiards.

For the Lorentz gas, $\musrb = \vf \mu$ where $\vf = c \cos \theta$ so that we are in the
setting of Theorem~E(a); however, $\vf = 0$ only when $\theta = \pm \pi/2$ so that
$\G_\vf = \E$ since the set $\{ \theta = \pm \pi/2 \}$ does not contain any invariant sets
by the finite horizon condition and so cannot contain the support of any invariant
measure.


\section{Ideas Common to the Proofs of Theorems A--C}
\label{proof a-c}

In this section, we first give the ideas common to the proofs of
Theorems A--C. Let $f$ be the mapping in question, let $m$ be the reference
measure (i.e. $m = \mu_\vf$ in Theorem A, $m=\musrb$ in Theorem B, and so on),
and let $\G$ be the relevant set of ergodic invariant measures with respect
to which the pressure term is defined (i.e. $\G = \G_H \cap \G_\vf$ in Theorem A,
and so on). This ``generic" notation is used throughout Sect.~\ref{proof a-c}.

If $\G = \emptyset$, then $\pa_{\G} = -\infty$ and
the theorem is vacuously true. Consider $\nu \in \G$.  Leaving precision for later,
our proof will proceed as follows: For $n \ge 0$, we introduce
dynamical balls in $M^n$ of the form
\[
B(x,n,g)  = \{ y \in M : d(f^ix, f^iy) < g(f^ix), 0 \leq i \leq n \} \cap M^n
\]
where $M^n =  \cap_{i=0}^n f^{-i}(M \setminus H)$ and
$g:M \to \mathbb{R}^+$ is a suitable function to be specified (think of
it as $g \approx \varepsilon$ for the moment). We will prove

\medskip
\noindent {\it I. Relation to entropy:} $\nu(B(x,n,g)) \sim
e^{-nh_\nu(f)}$.

\medskip
\noindent {\it II. Volume estimate:} $m(B(x,n,g)) \gtrsim e^{-n\lambda_\nu^+}$
where $\lambda_\nu^+$ is the sum of positive Lyapunov

exponents for $\nu$-a.e.\ $x$.

\medskip
\noindent From Estimate I, we deduce that $M^n$ contains
$\gtrsim e^{nh_\nu(f)}$ disjoint sets of the type $B(x,n,g)$.
This together with Estimate II gives
\begin{equation}
\label{M bound}
m(M^n) \gtrsim e^{nh_\nu(f)} \cdot e^{-n\lambda_\nu^+} .
\end{equation}
Taking $\log$, dividing by $n$ and letting $n \to \infty$, gives
$\underline \rho(m) \ge h_\nu(f) - \lambda_\nu^+$, which is what we need.

We now proceed to make these ideas precise.

\bigskip
\noindent {\bf I. Relation to entropy.} For this part
we cite the following very general result.

\begin{proposition}
\label{prop:local entropy}
Let $\Phi : X \circlearrowleft$ be a measurable transformation of a compact metric
space of finite capacity,
and let $\theta$ be an ergodic invariant measure for $\Phi$.
Let $\hat g_\ve$ be a family of functions satisfying $|\hat g_\ve|_\infty \leq \ve$ and
$\int_X -\log \hat g_\ve \, d\theta < \infty$, and define
$\hat B(x,n,\hat g_\ve)  = \{ y \in X : d(\Phi^ix, \Phi^iy) < \hat g_\ve(\Phi^ix), 0
\leq i \leq n \}$. Then for $\theta$-a.e. $x$,
\[
\lim_{\ve \to 0^+}  \liminf_{n \to \infty} - \frac 1n \log \theta(\hat B(x,n,\hat g_\ve))
= \lim_{\ve \to 0^+}  \limsup_{n \to \infty} - \frac 1n \log \theta(\hat B(x,n,\hat g_\ve))
= h_\theta(\Phi)  .
\]
\end{proposition}

Proposition \ref{prop:local entropy} follows from \cite[Lemma 2]{mane}
and \cite[Main Theorem]{brin katok}.
Note that although \cite{brin katok} is phrased in terms of a continuous map,
the proof does not use this fact.

\medskip
In the proofs of Theorems A--C, Proposition \ref{prop:local entropy} will
be applied with $\Phi=f$, $\theta = \nu \in \G$ and
\[
\hat g_{\ve}(x):= \min\{\ve,d(x,\Si)\}
\]
($\Si = \emptyset$ in Theorems A and B).  Observe that intersecting
$\hat B(x,n,\hat g_\ve)$ with $M^n$ does not
affect its $\nu$-measure since $\nu$ is supported on the survivor set.
From $\nu (N_\ve(\Si)) \leq C\ve^\alpha$, we have
\begin{equation}
\label{eq:g int}
\begin{split}
\int_M -\log (\hat g_{\ve}) \, d\nu
& \leq -\log \ve + \sum_{n=0}^\infty \nu(N_{\ve e^{- n}}(\Si)\setminus
N_{\ve e^{-(n+1)}}(\Si)) (n+1
	- \log \ve) \\
& \leq -\log \ve + \sum_{n=0}^\infty C\ve^\alpha e^{-\alpha n} (n+1
- \log \ve) < \infty
\end{split}
\end{equation}
so our $\hat g_{\ve}$ satisfies the hypotheses of Proposition~\ref{prop:local entropy}.

\bigskip
\noindent {\bf II. Volume estimate.} Let $g_\ve = \frac13 \hat g_\ve$.
Continuing to let $m$ denote the initial distribution  and $\nu \in \G$,
we state the following desired volume estimate:

\begin{proposition}
\label{prop:volume est}
There exists a measurable set $E \subset \Omega$ with $\nu(E) > 0$ such that
for $\nu$-a.e.~$x \in E$,
\begin{equation}\label{vol}
\sup_{\ve > 0} \limsup_{n \to \infty}
- \frac 1n \log m(B(x,n, g_{\ve}))
\leq \lambda_\nu^+ .
\end{equation}
\end{proposition}

\medskip
\noindent {\bf Proof of Theorems A--C assuming
Proposition \ref{prop:volume est}:} Let $\nu \in \G$ be given. We
fix $\delta > 0$, and let $\sigma := \nu(E)$ where
$E$ is as in Proposition~\ref{prop:volume est}.
Using Propositions~\ref{prop:local entropy}  and~\ref{prop:volume est},
we may choose first $\ve>0$ sufficiently small, and then $n_0 = n_0(\delta, \ve)
\in {\mathbb Z}^+$ sufficiently large
and a measurable set $E' \subset E$ with $\nu(E') \geq \sigma/2$
such that for every $x \in E'$,
\begin{enumerate}
  \item[(i)] $\nu(B(x,n, 3g_{\ve})) \leq e^{-n(h_\nu - \delta)}$ for all
  	$n \geq n_0$; \vspace{-6 pt}
  \item[(ii)] $m(B(x,n, g_{\ve})) \geq e^{-n(\lambda^+_\nu + \delta)}$ for all
  	$n \geq n_0$. \vspace{-2 pt}
\end{enumerate}

For $n \geq n_0$, let $\C_n \subset E'$ be a maximal set of points
such that $B(x_i,n, g_{\ve}) \cap B(x_j,n, g_{\ve})
=\emptyset$ whenever $x_i, x_j \in \C_n$, $x_i \neq x_j $.
By the maximality of $\C_n$, for every $y\in E'$, there exists $x_i\in \C_n$ such that
$B(y,n, g_\ve) \cap B(x_i, n, g_\ve) \neq \emptyset$. We will show momentarily
 that $y\in B(x_i,n,3g_{\ve}) $.
This will imply $E' \subset \cup_{x_i \in \C_n} B(x_i, n, 3g_{\ve})$, and hence
$|\C_n| \geq \frac \sigma 2 e^{n(h_\mu - \delta)}$ by (i).

To show $y\in B(x_i,n,3g_{\ve}) $, it suffices to show $d(f^ky, f^kx_i)<3g_\ve(f^kx_i)
\ \forall k\le n$,
since $y \in E' \subset M^n$. Now $B(y,n, g_\ve) \cap B(x_i, n, g_\ve) \neq \emptyset$
means there exists
$z \in M$ such that $d(f^kx_i, f^kz) \leq  g_\ve(f^kx_i)$ and $d(f^kz, f^ky) \leq
g_\ve(f^ky)$ for all $0 \leq k \leq n$. Thus the assertion above boils down to
the following lemma.

\begin{lemma}
\label{lem:distance_estimates}
For any $x,y \in M $, if there exists $z \in M$ with $d(x,z) \le g_{\ve}(x)$ and
$d(z,y)\le  g_{\ve}(y) $, then $d(x,y) \le 3g_{\ve}(x)$.
\end{lemma}

\begin{proof}[Proof of Lemma]
It suffices to show $g_\ve(y) \leq 2 g_\ve(x)$, for that will imply $d(x,y) \le
d(x,z) + d(z,y) \leq
g_{\ve}(x)+ g_{\ve}(y) \leq 3 g_{\ve}(x)$, proving the lemma.
Observe that
\begin{eqnarray*}
d(y,\Si) & \leq & d(y,z)+d(z,x)+d(x,\Si)\\
& \le  & g_{\ve}(y)+ g_{\ve}(x)+d(x,\Si) \ \leq \ \mbox{$\frac 13$} d(y,\Si)
+ \mbox{$\frac 43$} d(x,\Si),
\end{eqnarray*}
the last inequality following from $g_\ve(\cdot) \le \frac13 d(\cdot, \Si)$.
Altogether, this gives $d(y,\Si) \leq 2 d(x,\Si)$.

To finish, consider the following two cases: \\
Case 1:  $d(x, \Si) > \ve$.  With $g_\ve(x)=\frac13 \ve$,
$g_\ve(y)$ is automatically $< 2g_\ve(x)$ since it is $\le \frac13 \ve$. \\
Case 2: $d(x,\Si) \leq \ve$.  In this case $g_\ve(y) \leq \frac13 d(y, \Si) \leq
\frac23 d(x, \Si) = 2 g_\ve(x)$.
\end{proof}

For each $x\in E'$, we have $B(x,n,g_{\ve})\subset M^n$ by definition.  Since
the $B(x_i, n,  g_\ve)$ are disjoint, we may estimate $m(M^n)$ by
\[
m(M^n)  \; \geq \; \sum_{x_i \in \C_n} m(B(x_i, n,  g_\ve))
             \; \geq \; |\C_n|\cdot \min_{x_i \in \C_n} m(B(x_i, n,  g_{\ve}))
            \; \geq \; \frac \sigma 2  e^{n(h_\nu - \delta)} e^{-n(\lambda^+_\nu + \delta)} .
\]
This yields
\[
\liminf_{n \to \infty} \frac 1n \log m(M^n) \geq
h_\nu(f) - \lambda_\nu^+ - 2\delta .
\]
The theorem is proved since $\delta$ was chosen arbitrarily.
\hfill $\square$

\bigskip
To complete the proofs of Theorems A--C,  it remains only to prove
the volume estimate in Proposition~\ref{prop:volume est}.

\section{Volume Estimates}

In this section we prove Proposition~\ref{prop:volume est} in the various settings of interest.
The basic argument, which treats the case $\Si = \emptyset$,
$m= \mu$, and $\nu \in \G_H$
is presented in Sect. 4.1. Proofs of other cases in Theorems A--C are presented
as modifications of this one.


\subsection{Proof of Proposition~\ref{prop:volume est}: Basic setup}
\label{prop proofs}

We consider here the most basic setup, namely where  $\Si=\emptyset$,
$m=\mu$, and $\nu \in \G_H$ (as defined), and give a proof of Proposition~\ref{prop:volume est}.

\bigskip
\noindent {\it I. Plan.}
From the  pointwise nature of the result and the fact that
the quantity on the left of (\ref{vol}) increases as $\ve\to 0$, it suffices to
show that given $\kappa>0$, for $\nu$-a.e. $x$ and arbitrarily small $\ve>0$,
there exists $c(x,\ve)$ such that
$$
m(B(x,n, g_\ve)) \ge c(x, \ve) e^{-n(\lambda_\nu^+ + \kappa)} \qquad {\rm for \ all} \ n \ge 0\ .
$$
Here, $g_\ve(\cdot) \equiv \frac13 \ve$; remember that $B(x,n, g_\ve)$ is a dynamical
ball {\it in} $M^n$ (and not in $M$). Such an object is cumbersome to work with since it involves both
the dynamics and the hole. To remove the hole from consideration, we introduce
$$
B^*(x,n,\ve, \gamma) := \{y \in M: d(f^ix, f^iy) < \ve e^{-\gamma i} \ {\rm for} \ 0 \le i < n\}\ .
$$
By definition of $\G_H$, for any $\gamma>0$ and $\nu$-a.e. $x$, $B^*(x,n, \frac13 \ve, \gamma)
 \subset B(x,n, g_\ve)$ for small enough $\ve$. Thus it suffices to prove, for
 a suitably chosen $\gamma$
and arbitrarily small $\ve>0$,
\begin{equation}
\label{ball bound}
m(B^*(x,n, \ve, \gamma)) \ge c(x, \ve, \gamma) e^{-n(\lambda_\nu^+ + \kappa)}
\qquad {\rm for \ all} \ n \ge 0\ .
\end{equation}
This is what we will do. Our strategy is to make these volume estimates in
Lyapunov charts and pass them back to the manifold.

\bigskip
\noindent {\it II. Lyapunov charts and hyperbolic estimates.}
Let $\lambda_1 <  \ldots < \lambda_p$ be
the distinct Lyapunov exponents of $(f,\nu)$, with multiplicities
$m_1, \ldots, m_p$ respectively, and let $E_i(x)$ be the subspace of $T_xM$
corresponding to $\lambda_i$. For each $i$, we let $R_i(r)$ denote the ball
of radius $r$ centered at $0$ in $\mathbb R^{m_i}$, and let $R(r) =
\Pi_{i=1}^p R_i(r)$. We recall below the following facts about Lyapunov charts,
following the exposition in \cite{young notes}.

\begin{proposition} \cite[Sect.\ 3.1]{young notes}
\label{prop:lyapunov}
Let $\delta << \min_{i \ne j} |\lambda_i -\lambda_j|$ be fixed. Then
there is a measurable set $V' \subset M$, $\nu(V')=1$,
a measurable function $\ell : V'
\to [1, \infty)$ satisfying $\ell(f^\pm x)/\ell(x) < e^{2\delta}$, and
a family of charts $\{\Phi_x: R(\delta \ell(x)^{-1}) \to M\}_{x \in V'}$
with the following properties:
\begin{itemize}
\item[(a)] (i) $\Phi_x(0)=x$;

(ii) $D\Phi_x(\{0\} \times \cdots \times {\mathbb R}^{ m_i} \times \cdots \times \{0\})
= E_i(x)$;

(iii) for all $z,z' \in R(\delta \ell(x)^{-1})$,
$$
K^{-1} d(\Phi_xz, \Phi_xz') \le |z-z'| \le \ell(x) d(\Phi_xz, \Phi_xz')
$$
where $K$ is a constant depending only on the dimension of $M$.
\item[(b)] Let $\tilde f_x = \Phi_{fx}^{-1} \circ f \circ \Phi_x$ be
defined where it makes sense.  Then

(i) $e^{\lambda_i -\delta} |v| \le |D\tilde f_x(0)v| \le e^{\lambda_i +\delta} |v|$
for $v \in \{0\} \times \cdots \times {\mathbb R}^{ m_i} \times \cdots \times \{0\}$;

(ii) {\rm Lip}$(\tilde f_x - D\tilde f_x(0)) < \delta$;

(iii) {\rm Lip}$(D\tilde f_x) < \ell(x)$.
\end{itemize}

\end{proposition}

\medskip
The following notation is used: Let $T_xM = E^{cu}(x) \oplus
E^s(x)$ where $E^{cu}(x) = \oplus_{i: \lambda_i \ge 0} E_i(x)$ and $E^s(x) =
\oplus_{i: \lambda_i < 0} E_i(x)$.
We will estimate the volume of the sets in question by looking at slices
parallel to $E^{cu}$, and will do so in Lyapunov charts.
Let $R^{cu}$ and $R^s$ be the subspaces in the charts corresponding to
$E^{cu}$ and $E^s$, and let $R^{cu}(r)$ and
$R^s(r)$ denote disks of radius $r$ centered at $0$ in $R^{cu}$ and $R^s$
respectively.
We will work with compositions of chart maps, writing
$$
\tilde f^n_x := \tilde f_{f^{n-1}x} \circ \cdots \circ \tilde f_x\ ,
$$
and study graph transforms by $\tilde f^n_x$ of functions from
$R^{cu}(r)$ to $R^s(r)$. The precise assertions are as follows:

\bigskip
\noindent (a) For all $\gamma>0$ sufficiently small,
there exist $\delta, \sigma>0$ small enough and a chart system (with $\delta$
as in Proposition 4.1) such that the following holds for $\nu$-a.e. $x$:
Let $r \le \delta \ell(x)^{-1}$, and let $g_0: R^{cu}(r) \to R^{s}(r)$
be a $C^1$ function with $|g_0(0)|<\frac12 r$ and $\|Dg_0\| < \sigma$.
Then for $i=1,2, \cdots$, there exists $g_i: R^{cu}(e^{-\gamma i} r) \to
R^{s}(e^{-\gamma i} r)$ defined on exponentially shrinking domains and
with $\|Dg_i\| < \sigma$ for all $i$ such that inductively
$$
\tilde f_{f^{i-1}x}({\rm graph}(g_{i-1})) \cap R(e^{-\gamma i}r) \ = \
{\rm graph}(g_i)\ .
$$
That is to say, if $\Gamma_z$ is the graph transform by
$\tilde f_z$, then for each $i$, $\Gamma_{f^{i-1}x}(g_{i-1})=g_i$.

\bigskip
\noindent (b) For $g: R^{cu}(r) \to R^{s}(r)$ and  $y \in$ graph$(g)$,
let $T_g(y)$ denote the tangent space to the graph of $g$ at $y$.
If $\delta$ and $\sigma$ in (a) are small enough, then for
$y \in$ graph$(g_0)$ such that $\tf^i_x(y) \in R(e^{-\gamma i}r)$
for all $i\le n$,
$$|\det [D\tf^n_x(y)|_{T_{g_0}(y)}]| < e^{n(\lambda^+ + 3k \delta)}
$$
where $\lambda^+ = \sum_{i:\lambda_i >0} m_i \lambda_i$  and $k=$dim$(E^{cu})$.

\bigskip
Notice first that with $2\delta < \gamma$, we are assured that $R(e^{-\gamma i }r)$
lies in the chart at $f^ix$; this is because $\ell(f^ix) > e^{-\gamma i}\ell(x)$;
see Proposition 4.1. Since most of the other assertions in
(a) and (b) follow from standard (uniformly hyperbolic) graph transform
estimates, we will only sketch the arguments for a few key points. (A version of
these estimates can be found in \cite[Sect.\ 3.1]{young notes}; see also
\cite[Sect.\ B]{young large d} for similar results.)

The ``overflowing property" of the graph transforms can be justified as follows.
Consider first the case where $g_0(0)=0$.
By Proposition 4.1(b)(i), $|D\tf_x(0)v| \ge
e^{-\delta} |v| \approx (1-\delta) |v|$ for $v \in R^{cu}$.
By Proposition 4.1(b)(iii) together with chart size, we have, for all $\eta \in R(r)$,
$$|D\tf_x(\eta)v - D\tf_x(0)v| \le {\rm Lip}(D\tf_x) r |v| < \delta |v|\ .
$$
This gives $|D\tf_x(\eta)v| > (1-3\delta) |v| \approx e^{-3\delta}|v|$
for $v$ with a small enough component in $R^s$. Thus with $\delta$
and $\sigma$ sufficiently small relative to $\gamma$, the overflowing property
is assured from step to step for $g_0$ with $g_0(0)=0$.
For graphs that do not pass through $0$, we pivot them at
$y \in \tilde W^s_{loc} \cap$graph$(g_0)$ where $\tilde W^s_{loc}$ is the
stable manifold of $x$ in its chart. Since $|\tf_x^i(y)| < e^{(\lambda_s + \delta) i}$
where $\lambda_s = \max\{\lambda_i: \lambda_i < 0\}$, movements of
$\tf^i_x(y)$ in the $R^{cu}$-direction are negligible assuming
$\gamma << |\lambda_s|$.

The assertion in (b) is proved similarly: We view $\det(D\tf_x^n)$
as a product of determinants. At each step,
$|\det(D\tf_{f^ix}(0)|_{R^{cu}})| < e^{\lambda^+ + k\delta}$, and
we may assume that approximations of the type in the last paragraph
increase the error by a factor $<e^{2k\delta}$.

\bigskip
\noindent {\it III. Completing the proof.}
Putting assertions (a) and (b) in II together, we arrive at the following:
Define
$$
\tilde B^*(x,n,r,\gamma) = \{y\in R^{cu}(r) \times R^s(\mbox{$\frac 12 r $}):
\tf^i_x(y) \in R(e^{-\gamma i} r) {\rm \ for} \ i=1,2, \cdots, n\}\ .
$$
We foliate $\tilde B^*(x,0,r,\gamma)$ with planes $\{P\}$ parallel to $R^{cu}(r)
\times \{0\}$, and view them as graphs of constant functions. By the overflowing
property of the graph transform at each step, $\tf^i_x(y) \in R(e^{-\gamma i} r)$
for all $i \le n$ is equivalent to $\tf^n_x(y) \in R(e^{-\gamma n} r)$.
Pulling back $\tf^n_x(P) \cap R(e^{-\gamma n}r)$, we use the bound
in assertion (b) to estimate the area of $P \cap \tilde B^*(x,n,r,\gamma)$.
We then integrate over $\{ P \}$ to obtain

\begin{equation}
\label{bound2}
{\rm Leb}(\tilde B^*(x,n,r,\gamma)) \ge \Big( \frac r2 \Big)^{d-k} \cdot
(r e^{-\gamma n})^k \cdot e^{-n(\lambda^+ + 3k \delta)}\
\end{equation}
where $d=$dim$(M)$.

We now return to the argument outlined at the beginning of the proof.
Let $\kappa>0$ be given. Assuming always $\gamma << |\lambda_s|$,
we now take it small enough that $4k\gamma < \kappa$, and let $\delta$ be
small enough (with respect to $\gamma$) for assertions (a) and (b) to hold
in the chart system $\{\Phi_x\}$ associated with $\delta$. For $\nu$-typical $x$,
we consider $\ve$ small enough that $B(f^ix, \ve e^{-\gamma i}) \cap H
=\emptyset$ for all $i\ge 0$.  Choosing $r< \ve/K$
where $K$ is as in Proposition 4.1(a)(iii), we define $\tilde B^*(x,n,r,\gamma)$
in the chart at $x$ as above, and observe that
$\Phi_x(\tilde B^*(x,n,r,\gamma)) \subset B^*(x,n, \ve, \gamma)$.
To finish, it remains to pass the estimate in (\ref{bound2}) back to $M$.
Proposition 4.1(a)(iii) gives  a bound on the Jacobian of $\Phi_x$,
allowing us to conclude
$$
m(\Phi_x(\tilde B^*(x,n,r,\gamma)) \ge
\ell(x)^{-d} \cdot  {\rm Leb}(\tilde B^*(x,n,r,\gamma)) \ge
c \, e^{-n(\lambda^+ + \kappa)}
$$
for some constant $c$ depending on $x, \ve$ and $\gamma$.
\hfill $\square$

\subsection{Adaptations of basic argument to various settings}

We now explain how each of the other results in Theorems A--C is deduced
from the proof in Sect. 4.1.

\bigskip
\noindent {\bf 1. The $W^s$-neighborhood condition (O):}
We continue to assume $\Si=\emptyset$ and $m=\mu$. To relax the condition
from the original definition of $\G_H$ in Sect. \ref{upper bound} to the one given by {\bf (O)}, the proof in Sect. 4.1 is modified as follows:
Given $\kappa$, we fix $\gamma, \delta$, a chart system $\{\Phi_x\}$, and
a $\nu$-typical $x \in M$. Let $O$, a $W^s$-neighborhood, and $\ve$ be
such that $f^i(O) \cap B(f^ix, \ve e^{-\gamma i}) \subset M
\setminus H$ for all $i$. We need to show $m(O \cap B^*(x,n,\ve,\gamma))
\ge c(x,\ve, \gamma) e^{-n(\lambda^+_\nu + \kappa)}$. Let $r< \ve/K$.

The following notation is used: For $y \in R(r)$ and small $\eta>0$,
let $R(y,\eta)=y+R(\eta)$; if $y=(y^{cu}, y^s)$ are the coordinates of $y$
with respect to $R^{cu}$ and $R^s$,
we write $R^{cu}(y^{cu}, \eta)= y^{cu}+R^{cu}(\eta)$, and so on.
To define the analog of $\tilde B^*(x,n,r,\gamma)$ in Sect. 4.1,
let $z \in O \cap W^s_{loc}(x)$ be sufficiently close to $x$, let $\tilde z :=\Phi_x^{-1}(z)$,
and let $r'< r$ be small enough that $\Phi_x(R(\tilde z, r')) \subset O$.
Define $\tilde B^*(x,n, z, r',\gamma)$
$$
 :=   \{y\in R^{cu}(\tilde z^{cu}, r') \times
R^s(\tilde z^s, \mbox{$\frac12 r'$}):
\tf^i_x(y) \in R(\tf^i_x(\tilde z), e^{-\gamma i} r') {\rm \ for} \ i=1,2, \cdots, n\}\ .
$$
Since $z \in W^s_{loc}(x)$, $\tf^i_x(\tilde z) \to 0$ as $i \to \infty$. It is straightforward to check that modulo a constant,
Leb$(\tilde B^*(x,n, z, r',\gamma))$ is bounded below by the quantity on
the right side of (\ref{bound2}),
and that $\Phi_x(\tilde B^*(x,n, z, r',\gamma)) \subset
(O \cap B^*(x,n,\ve,\gamma))$.
\hfill $\square$

\bigskip
In the settings below, we will revert back to $\G_H$ as defined,
leaving it to the reader to extend the proof to include the condition {\bf (O)}
if they so choose.

\bigskip
\noindent {\bf 2. Initial distributions with densities:}
Continuing to assume $\Si = \emptyset$,
we let $m=\mu_\vf $ for some $\vf \in L^1(\mu)$. Let $\nu \in \G_H \cap
\G_\vf$, and let $Z$ and $c_\nu$ have the meaning in the definition of $\G_\vf$.
Observe that for $\nu$-a.e. $x \in Z$ and
small enough $\ve$, one has $\mu_\vf (B^*(x,n,\ve, \gamma)) \ge
c_\nu m(B^*(x,n,\ve, \gamma))$.
An argument identical to that in Sect. 4.1 proves Proposition~\ref{prop:volume est}
with $E=Z$. \hfill $\square$

\bigskip
\noindent {\bf 3. SRB measures as initial distributions:} Continuing to assume
$\Si = \emptyset$, we let $m=\musrb$ as in Theorem B. Given $\nu \in \G_H \cap
\Gsrb$, we fix a $\musrb$ hyperbolic product set
$\Pi = (\cup \Gamma^u) \cap (\cup \Gamma^s)$
with $\nu(\Pi)>0$, and show that the volume
estimate for $m(B^*(x,n,\ve,\gamma))$ in Sect. 4.1 holds for $\nu$-a.e. $x \in \Pi$.

Let $x \in \omega^u_x \cap \omega^s_x \in \Pi$ be a $\nu$-typical point, where
$\omega^u_x \in \Gamma^u$ and $\omega^s_x \in \Gamma^s$.
Note that due to the uniform contraction of $T^n\omega^s_x$ and $T^{-n}\omega^u_x$
required by (W.1) of Sect.~\ref{upper bound}.B, $x$ can have no zero Lyapunov
exponents.
Let $\tilde W^u_{loc}$ denote the image
of the local unstable manifold through $x$ in its chart.
Since local unstable manifolds are unique (up to size),
$\Phi_x^{-1}(\omega^u_x) \subset \tilde W^u_{loc}$, which has the dimension of
$R^u$ and is tangent to it at $0$. (Since no zero Lyapunov exponents is an
assumption for Theorem B, we have $R^u$ instead of $R^{cu}$.)
By conditions (W.1) and (W.2),
for all small enough $r>0$, there exists $\Gamma_0 \subset \Gamma^u$ such that
(i) $\musrb(\cup_{\omega \in \Gamma_0} \omega)>0$ and
(ii) for every $\omega \in \Gamma_0$, $\Phi_x^{-1}(\omega) \cap R(r)$ is the graph of
a function from $R^u(r)$ to $R^s(r)$ with the properties of $g_0$ in Sect. 4.1.
Define
$$
\tilde B^*(x,n,r, \Gamma_0, \gamma) =
\{y \in \cup_{\omega \in \Gamma_0} \Phi_x^{-1}\omega :
\tf^i_x(y) \in R(e^{-\gamma i} r) {\rm \ for} \ i=0,1,2, \cdots, n\}\ .
$$
With $r$ small enough relative to $\ve$, clearly
$\Phi_x(\tilde B^*(x,n,r, \Gamma_0, \gamma)) \subset B^*(x,n,\ve,\gamma)$.
To estimate the measure of this set, it is more convenient to bring
$\musrb$ to the chart (instead of doing it on $M$): Let $\alpha$ be the measure
$(\Phi_x^{-1})_*(\musrb|_{\cup_{\omega \in \Gamma_0} \omega})$
restricted to $R(r)$. By (i) above together with (W.3),
$\alpha(\tilde B^*(x,0,r, \Gamma_0, \gamma))>0$.
We disintegrate $\alpha$
into conditional probability measures on the leaves $\{\Phi_x^{-1}\omega\}$,
letting $\alpha_T$ denote the measure in the transverse direction.
To estimate the $\alpha$-measure of $\tilde B^*(x,n,r, \Gamma_0, \gamma)$,
we do it one $\Phi_x^{-1}\omega$-leaf at a time, integrating
with respect to $\alpha_T$ afterwards. Condition (W.3) ensures
uniform lower bounds of the type in (\ref{bound2}) for $\alpha_T$-almost all leaves.
\hfill $\square$

\bigskip
\noindent {\bf 4. Maps with singularities:} We discuss the case $m=\mu$, leaving the
others to the reader.

\smallskip
Let $\nu \in \G_H \cap \G_\Si$, and observe
the following lemma.

\begin{lemma}
\label{lem:approach}
Let $E_{\ve,\gamma} = \{ x \in M : d(f^ix, \Si)>\ve e^{-\gamma i}$ for all $i\ge 0 \}$.
Then for any fixed $\gamma>0$, $\lim_{\ve \to 0} \nu(E_{\ve,\gamma}) = 1$.
\end{lemma}
\begin{proof}
This follows from the simple estimate,
\[
\nu(M \setminus E_{\ve,\gamma}) = \sum_{i\ge 0} \nu [f^{-i}( N_{\ve e^{-\gamma i}}(\Si))]
= \sum_{i\ge 0} \nu [N_{\ve e^{-\gamma i}}(\Si)] \le
\sum_{i \ge 0} C \ve^\alpha e^{-\gamma \alpha i} \le C' \ve^\alpha\ ,
\]
the first inequality coming from the definition of $\G_\Si$.
(See also \cite[Part I, Lemma 3.1]{katok}.)
\end{proof}

This means that for $x \in E_{\ve, \gamma}$, we again have
$B^*(x,n, \frac13 \ve,\gamma) \subset B(x,n, g_{\ve})$,
for $g_{\ve}(f^ix) = \frac13 \min\{\ve, d(f^ix, \Si)\} \ge \frac13 \ve e^{-\gamma i}$.

Continuing to follow the proof in Sect.\ 4.1, we note that
the definition of $\G_\Si$ together with \eqref{eq:d1 blowup} implies that
$\int_M \log^+ \| Df^{\pm 1}_x \| \, d\nu < \infty$ where $\log^+ x = \max \{ \log x, 0 \}$,
so Lyapunov exponents are well defined $\nu$-a.e.  In addition, the Lyapunov charts
described in Proposition 4.1 exist for this class of maps with some modifications
due to the presence of singularities.

Observe first that $(a)(i)$, $(a)(ii)$ and $(b)(i)$ of Proposition 4.1 hold as stated
since these quantities depend only
on $Df$ at a typical point $x$ (see
\cite[Part I, Theorem 2.2]{katok}).\footnote{Although \cite{katok} uses only
a single splitting, $T_xM = E_\alpha(x) \oplus E_\beta(x)$, one can just as easily
split the tangent space into $\oplus_i E_i(x)$, one for each Lyapunov exponent,
to obtain Proposition 4.1$(b)(i)$ using an argument identical to that in
\cite[Sect.\ 3.1]{young notes}.}

The other items of Proposition~\ref{prop:lyapunov} are modified as follows.  Fix $\delta$ as in Proposition~\ref{prop:lyapunov}.  Then there exist a set $V'$ with $\nu(V')=1$
and a measurable function $\ell(x):V' \to [ 1, \infty)$, with $\ell(f^{\pm}x) < e^{2\delta} \ell(x)$,
such that
for all $\ve>0$ sufficiently small, the charts $\Phi_x$ are defined on $R(\delta \ell(x)^{-1} g_\ve(x)^b)$, where $b$
is the exponent from \eqref{eq:d2 blowup}, and satisfy

\begin{itemize}

  \item[$(a)$] $(iii')$ For all $z,z' \in R(\delta \ell(x)^{-1} g_\ve(x)^b)$,
\[
K^{-1} d(\Phi_xz, \Phi_xz') \le |z-z'| \le \ell(x) d(\Phi_xz, \Phi_x z')
\]
where $K$ is a constant depending only on the dimension of $M$ and $c_0$ from
\eqref{eq:exp}.

\item[$(b)$] Let $\tilde f_x = \Phi_{fx}^{-1} \circ f \circ \Phi_x$ be defined where it makes sense.  Then

$(ii')$ {\rm Lip}$(\tilde f_x - D\tilde f_x(0)) \le \delta$;

$(iii')$ {\rm Lip}$(D\tilde f_x) \le \ell(x) g_\ve(x)^{-b}$.
\end{itemize}
Although the construction of these charts is similar to that found in
\cite{pesin 1, katok}, we include the necessary arguments in the Appendix
since the statements we need are somewhat different from those found in
the literature.

With the charts $\{\Phi_x\}$ in place, the proof follows a similar line to that given
in Section 4.1, with slight modifications due to the singularities.  For example,
assertion (a) is no longer a uniform statement for all $x \in V'$; rather, we need to
choose $r \le \delta \ell(x)^{-1}g_\ve(x)^b$, but only
after $\ve$ is fixed depending on the rate of approach of $x$ to the singularities.
We state precisely these changes below.

Fix  $\kappa >0$ and choose
$\gamma << |\lambda_s|$ such that $(b+4)k\gamma < \kappa$.
Using Lemma~\ref{lem:approach}, we choose
$\ve>0$ such that $\nu(E_{\ve, \gamma}) > 1-\kappa$.
Next we choose $\delta>0$ with $2\delta < \gamma$,
so that there exists a chart system $\{ \Phi_x \}_{x \in V'}$
with the modified properties as listed in $(a)(iii')$-$(b)(iii')$ above.
Note that $\nu(V' \cap E_{\ve,\gamma}) > 1- \kappa$.

We now choose $x \in V' \cap E_{\ve, \gamma}$ and prove the estimate \eqref{ball bound}.
Note that $B(f^ix, \ve e^{-\gamma i}) \cap (H \cup \Si) = \emptyset$.
Finally, choosing $r \le \delta \ell(x)^{-1} g_\ve(x)^b$
guarantees that
the assertions (a) and (b) of Sect.~4.1 hold along
the orbit of $x$ with $\gamma$ replaced by $\gamma(b+1)$,
for then $R(r e^{-i \gamma (b+1)})$ lies in the chart at $f^ix$
by definition of $E_{\ve, \gamma}$ and choice of $r$.
In particular, $\tf_{f^ix}$ is defined on
$R(r e^{-i \gamma (b+1)})$.
We shrink $r$ further if necessary so that $r< \ve/(3K)$
and define $\tilde B^*(x,n,r,\gamma(b+1))$ as in Sect.~4.1.
Then by item $(a)(iii')$ above,
$\Phi_{x}(\tilde B^*(x,n,r,\gamma(b+1))) \subset B^*(x,n,\frac 13 \ve, \gamma)$
and the rest of the
proof follows line by line with only minor changes to constants.
For example, \eqref{bound2} has the factor $(r e^{-\gamma n (b+1)})^k$ as
indicated above.

This proves Proposition~\ref{prop:volume est} for all $x \in E_{\ve, \gamma}$.  But since
$\kappa>0$ was chosen arbitrarily, by Lemma~\ref{lem:approach} we conclude that
Proposition~\ref{prop:volume est} holds for $\nu$-a.e. $x$.
\hfill $\square$

\section{Towers with Holes}
\label{tower}

This section is exclusively about escape dynamics on towers.
Sect.~\ref{tower review} reviews basic facts and notation for towers
making precise {\bf (A.1)}--{\bf (A.4)} in Sect.~\ref{tower results}.
In Sect.~\ref{tower theorems} we formulate results
analogous to Theorems D and E for {\it towers with Markov holes}.
Proofs are given in Sects.~5.3 and 5.4.

\subsection{Review of definitions and basic facts}
\label{tower review}

\noindent {\bf I. Closed systems (without holes)}

\medskip
Let $f:M\circlearrowleft$ is a
(piecewise) $C^{1+\epsilon}$ diffeomorphism.
The material below is taken from \cite{young tower}. We recall only essential definitions,
referring the reader to \cite{young tower} for detail.

\medskip
\noindent {\bf Generalized horseshoes:}  The idea of a {\it  generalized horseshoe with infinitely many branches and variable return times}, denoted $(\Lambda, R)$,
is as follows: $\Lambda \subset M$ is a compact subset with a hyperbolic product
structure, {\it i.e.}, $\Lambda = (\cup \Gamma^u) \cap (\cup
\Gamma^s)$ where $\Gamma^s$ and $\Gamma^u$ are continuous families
of local stable and unstable manifolds, and $\mu_{\omega}\{\omega
\cap \Lambda\}>0$ for every $\omega \in \Gamma^u$ where $\mu_\omega$
is the Riemannian measure on the unstable manifold $\omega$. We say
$\Lambda^s$ is an $s$-subset of $\Lambda$ if
$\Lambda^s= (\cup \Gamma^u) \cap (\cup \tilde \Gamma^s)$ for some
$\tilde \Gamma^s \subset \Gamma^s$, and $u$-subsets are defined similarly.
Modulo a set the restriction of which to each $\omega \in \Gamma^u$ has
$\mu_\omega$-measure zero, $\Lambda$ is a countable
disjoint union of (closed) $s$-subsets $\Lambda_j$ with the property
that for each $j$,
there exists $R_j \in {\mathbb Z}^+$ such that $f^{R_j}(\Lambda_j)$ is a
$u$-subset of $\Lambda$. The function $R: \Lambda \to {\mathbb Z}^+$
given by $R|_{\Lambda_j} =R_j$ is called the {\it return time function} to $\Lambda$.

The definition of a generalized horseshoe includes conditions on
hyperbolicity formulated as {\bf (P1)}--{\bf (P5)} in \cite{young tower}.
We will omit them and
focus instead on the estimates derived from these conditions that we will need.
Let $\omega^s(x)$ and $\omega^u(x)$ denote respectively
the elements of $\Gamma^s$ and $\Gamma^u$ containing $x$.
\begin{itemize}
\item[$\bullet$] There is a {\it separation time} $s: \Lambda \to
{\mathbb Z}^+$ with the property that (i) $s(x,y) = s(x',y')$ for
$x' \in \omega^s(x), y' \in \omega^s(y)$; (ii) for $x, y \in \Lambda_j$,
$s(x,y) \ge R_j$, and (iii) for $x \in \Lambda_j$, $y \in \Lambda_{j'}$,
$j \ne j'$, $s(x,y) \le \min (R_j, R_{j'})$.
\item[$\bullet$] There are constants $C>0$ and $\alpha \in (0,1)$, related to the
hyperbolicity and distortion of $f$, such that if $y \in \omega^s(x)$, then $d(f^nx, f^ny) \le C \alpha^n$
for all $n \ge 0$.
\end{itemize}

The following facts about the Jacobian in the unstable direction are useful.
For $\omega, \omega'  \in \Gamma^u$, the holonomy map
$\Theta_{\omega, \omega'}:
\omega \cap \Lambda \to \omega' \cap \Lambda$ is
obtained by sliding along stable curves, i.e.
$\Theta_{\omega, \omega'}(x) = \omega^s(x) \cap \omega'$.
Fix an arbitrary leaf $\hat \omega \in \Gamma^u$. We
let $\hat \Theta(x)$ be the unique point in $\omega^s(x) \cap \hat \omega$,
and define $a(x) = \log \prod_{i=0}^\infty \frac{\det Df^u(f^ix)}
{\det Df^u(f^i(\hat \Theta x))}$,
where $\det Df^u(x) = \det (Df(x) |_{E^u(x)})$ is the unstable Jacobian of $f$.
This function is used to define a
family of reference measures $\{m_\omega, \omega \in \Gamma^u\}$,
where
$m_\omega$ is the measure on $\omega$ whose density with respect to
$\mu_\omega$ is $e^a \cdot 1_{\omega \cap \Lambda}$.
For $x \in \omega \cap \Lambda_i$,
let $\omega'$ be such that $f^{R_i}(\omega \cap \Lambda_i) = \omega'$.
We define
 $J^u(f^R)(x) = J_{m_\omega, m_{\omega'}}(f^{R_i}|
(\omega \cap \Lambda_i))(x)$, the Jacobian of $f^R$ with respect to
the measures $m_\omega$ and $m_{\omega'}$.

\medskip
{\it Remark on notation:} It is convenient in this section to follow the
notation in [Y3], some of which conflicts, however, with earlier notation.
For example, $m$ in the last paragraph is not intended
to signify any relation to initial distributions in escape dynamics,
and $C_1$ below is not related to the same notation in Sect. 2.1,
Paragraph III. We do not believe this will lead to problems
as the contexts are quite different.

\begin{lemma} (\cite[Lemma 1]{young tower})
\label{lemma:jacobian}
\ (1) For all $\omega$, $\omega' \in \Gamma^u$,
  $(\Theta_{\omega,\omega'})_*m_\omega = m_{\omega'}$.
\begin{itemize} \vspace{-3 pt}
  \item[(2)] For each $\omega \in \Gamma^u$ and $x \in \omega$,
  $J^u(f^R)(x) = J^u(f^R)(y)$ for all $y \in \omega^s(x)$. \vspace{-6 pt}
  \item[(3)] $\exists C_1>0$ (depending on $C$ and $\alpha$) such that
  for each $\omega \in \Gamma^u$, $i \in \mathbb Z^+$ and all
  $x,y \in \Lambda_i \cap \omega$,
    \begin{equation}\label{eq:C_1}
        \left| \frac{J^u(f^R)(x)}{J^u(f^R)(y)} - 1\right|
    \leq C_1 \alpha^{s(f^Rx, f^Ry)/2} .
    \end{equation} \vspace{-3 pt}
\item[(4)] $\sup_{x \in \Lambda} a(x) < \infty$ and $|a(x)-a(y)| \le
 4C \alpha^{\frac12 s(x,y)}$ on each $\omega \in \Gamma^u$.
 \end{itemize} \vspace{-6 pt}
\end{lemma}

We say $(\Lambda, R)$ has {\it exponential
return times} if there exist $C_0>0$ and $\theta_0>0$ such that for
all $\omega \in \Gamma^u$, $\mu_\omega\{R>n\} \le C_0 \theta_0^n$
for all $n \ge 0$. This property (in fact, integrability of $R$ is sufficient)
plus the requirement that g.c.d.$\{R\} = 1$
guarantees that $f$ has a unique SRB measure $\musrb$ with $\musrb(\Lambda)>0$
(\cite[Theorem 1]{young tower}).

\medskip
\noindent {\bf ``Hyperbolic'' Markov towers:} Given $f$ with a generalized horseshoe
$(\Lambda, R)$, it is shown in \cite{young tower} that one can associate
a {\it Markov extension}
$F: \Delta \to \Delta$ which focuses on the return dynamics to $\Lambda$.
The set $\Delta$ is the disjoint union
$\cup_{\ell \ge 0} \Delta_\ell$ where
$\Delta_\ell$, the $\ell^{\mbox{\tiny th}}$ level of the tower, is defined
to be $\Delta_\ell = \{ (x, \ell) : x \in \Lambda, R(x) > \ell \}$, and
$F$ is defined by $F(x,\ell) = (x, \ell+1)$ for $\ell < R(x)-1$ and
$F(x,\ell) = (f^Rx, 0)$ when $\ell = R(x)-1$; that is to say, $F$ maps
$(x,0)$ successively up the tower until the return time for $x$ is reached.
A projection $\pi: \Delta \to M$ with $\pi \circ F = f \circ \pi$ is uniquely
defined assuming the natural identification of $\Delta_0$ with $\Lambda$.

For notational simplicity, we will often refer to
a point in $\Delta$ as $x$ when the level $\ell$ is made clear by context.

The separation function $s(\cdot, \cdot)$ above defines a countable partition
$\{\Delta_{\ell,j}\}$ on $\Delta$: for $x,y \in \Delta_0$,
$s(x,y) = \inf \{ n>0: \mbox{$F^nx, F^ny$ lie in
different $\Delta_{\ell,j}$} \}$. It is easy to see that $\{\Delta_{\ell,j}\}$ is a
Markov partition for $F$ with $\Delta_0$ as a single element.
Let $\dlj^* = \dlj \cap F^{-1}\Delta_0$. Note that
$F|_{\dlj^*}$ maps $\dlj^*$ bijectively onto a $u$-subset of $\Delta_0$,
and if we rename the collection $\{F^{-\ell}\dlj^*\}$ as $\{(\Delta_0)_i\}$,
then $\{(\Delta_0)_i\}$  is a countable collection of closed subsets of
$\Delta_0$ the $\pi$-images of which are precisely the $\{\Lambda_i\}$ in
the paragraph on generalized horseshoes.

Stable and unstable sets for $\Delta_{\ell, j}$ are defined as follows:
Let $\Gamma^s(\pi(\Delta_{\ell, j}))$ and
$\Gamma^u(\pi(\Delta_{\ell, j}))$ be the stable and unstable families
defining the hyperbolic product set $\pi(\Delta_{\ell, j})$. We say
$\tilde \omega \subset \Delta_{\ell, j}$ is an {\it unstable set} of
$\Delta_{\ell, j}$ if $\pi(\tilde \omega) = \omega \cap \pi(\Delta_{\ell, j})$
for some $\omega \in \Gamma^u(\pi(\Delta_{\ell, j}))$. Since there can be
no ambiguity, we will use $\Gamma^u(\Delta_{\ell, j})$ to denote the set
of all such $\tilde \omega$,
and let $\Gamma^u(\Delta) = \cup_{\ell,j} \Gamma^u(\Delta_{\ell, j})$.
{\it Stable sets} of $\Delta_{\ell, j}$ and $\Gamma^s(\Delta)$ are defined similarly.

Two reference measures $\tilde \mu_\omega$ and $\tilde m_\omega$ are
defined on $\omega \in \Gamma^u(\Delta)$ as follows: On $\Delta_0$,
identifying $\omega \in \Gamma^u(\Delta_0)$ with
$\omega' \cap \Lambda$ for $\omega' \in \Gamma^u(\Lambda)$,
$\tilde \mu_\omega$ is simply $\mu_{\omega'}|_{\omega' \cap \Lambda}$
and $\tilde m_\omega$ is $m_{\omega'}$. Once these measures are defined
on $\omega \in \Gamma^u(\Delta_0)$, there is exactly one way to
extend them to $\cup_{\ell>0} \Gamma^u(\Delta_\ell)$
so that if $J^u_\mu(F)$ and $J^u(F)$ denote the Jacobians of $F$ on
unstable sets with respect to $\tilde \mu_\omega$ and $\tilde m_\omega$
respectively, then $J^u(F) = J^u_\mu(F) = 1$ on $\Delta \setminus F^{-1}(\Delta_0)$.
Notice also that if we extend $a$ to $\cup_{\ell>0} \Delta_\ell$ by
$a(x)=a(F^{-1}x)$, then $d\tilde m_\omega = e^a \, d\tilde \mu_\omega$
on all $\omega \in \Gamma^u(\Delta)$.

\medskip
\noindent {\bf Quotient ``expanding" towers:}  Associated with
$F:\Delta \to \Delta$ is a quotient tower $\barF: \bDelta \to \bDelta$
obtained by collapsing stable sets
to points, i.e., $\bDelta = \Delta/\! \!\sim$ where for
$x,y \in \Delta$, $x \sim y$ if and only if $y \in \omega (x)$ for some
$\omega \in \Gamma^s(\Delta)$. Let
$\bpi: \Delta \to \bDelta$ be the projection defined by $\sim$.
We will use the notation
 $\bDelta_\ell =\bpi (\Delta_\ell), \bDelta_{\ell,j} = \bpi(\dlj)$, and so on.

Lemma~\ref{lemma:jacobian}(1) and (2) together imply that
there is a natural measure $\bm$ on
$\bDelta$ with respect to which  the Jacobian of $\barF$,
$J\barF$, is well defined: specifically, we have
$J\barF \equiv 1$ on $\bDelta \setminus
\barF^{-1}(\bDelta_0)$, and for $ x \in \bDelta_0$, $J\barF^R(x) = J^u(f^R)(y)$ for any $y \in \omega^s(x)$.
Finally, with the definition of separation time inherited
from $\Delta_0$, the distortion bound in Lemma 5.1(3) holds for
$J\barF^R$ on $\bDelta_0$.

\bigskip
\noindent {\bf II. Systems with holes}

\medskip
The setting is as in Paragraph I. We fix an open set $H \subset M$ and call it ``the hole."

\medskip
\noindent {\bf Towers with Markov holes} (following [DWY]):
Let  $(F, \Delta)$ be the tower arising
from the horseshoe $(\Lambda, R)$. We say $(F, \Delta)$
{\em respects the hole} $H$ if the following conditions are satisfied:
\begin{enumerate}  \vspace{-3 pt}
	\item[{\bf (H.1)}]  $\pi^{-1} H$ is the union of countably many elements of
	$\{\Delta_{\ell,j}\}$.  \vspace{-6 pt}
	\item[{\bf (H.2)}]  $\pi(\Delta_0) \subset M\setminus H$, and there exist
	$\delta > 0$, $\xi_1 >1$ such that all
	$x \in \pi(\Delta_0)$ satisfy $d(f^nx, \Si \cup \partial H) \geq \delta \xi_1^{-n}$ for all
	$n \geq 0$.   \vspace{-3 pt}
\end{enumerate}	
Because of {\bf (H.1)}, we refer to $\pi^{-1}H$, the hole on $\Delta$, as
a ``Markov hole."  This implies in particular that for every $i$ and $\ell$
with $0\le \ell < R_i$, $f^\ell(\Lambda_i)$ either does not meet $H$ or it
is completely contained in $H$. Equivalently, on the tower $(F, \Delta)$,
each $(\Delta_0)_i$ either falls into the hole completely on its way up the tower
or returns to $\Delta_0$ intact.

Earlier on we have used $(f,M;H)$ to denote an open system. Observe that
$(F, \Delta; \pi^{-1}H)$ and $(\overline F, \bDelta; \bH)$ where
$\bH = \bpi(\pi^{-1}H)$ are open systems of the same type. As before, we write
$$\Delta^n = \cap_{i=0}^n F^{-i} (\Delta \setminus \pi^{-1}H)
= \{ x \in \Delta: F^ix \notin \pi^{-1} H \mbox{ for } 0 \le i \le n \}\ ,$$
and $\Delta^\infty = \cap_{i=0}^\infty \Delta^n$. In particular,
$\Delta^0 = \Delta \setminus \pi^{-1}H$. The notation
$\F^n = F^n|_{\Delta^n}$ for $n \geq 1$ is sometimes used to distinguish
between the system with and without holes. Corresponding objects for $(\barF, \bDelta; \bH)$ are denoted by $\bDelta^n$ and $\bDelta^\infty$ etc.

\bigskip
\noindent {\bf III. Abstract towers and a notion of spectral gap}

\medskip
In Paragraphs I and II, we considered towers that arise from generalized
horseshoes. Towers can, in fact, be defined in the abstract.
Leaving details to the reader, an abstract expanding tower is  a dynamical system
$\barF : \bDelta \to \bDelta$ where $\bDelta_0$ is a compact set,
$\bDelta = \cup_{\ell \ge 0} \bDelta_\ell$
has a tower structure, $\barF$ moves points up the tower until their
return time $R$; there is a  countable Markov partition
$\{\bDelta_{\ell,j}\}$ on $\bDelta$ which is a generator
and a reference measure $\bm$
with respect to which we have (i) $J\barF = 1$ on $\bDelta \setminus
\barF^{-1}(\bDelta_0)$ and (ii) modulo a set of $\bm$-measure zero,
$\bDelta_0 = \cup_i (\bDelta_0)_i$ where $\barF^R$ maps
each closed set $(\bDelta_0)_i$
homeomorphically onto $\bDelta_0$
with the distortion bound
in Lemma 5.1(3).
Abstract expanding towers with Markov holes
$\bH$ are defined in the obvious way, as are abstract hyperbolic towers.

Given $(\barF, \bDelta)$ with $\bm\{R>n\} < C_0 \theta_0^n$
for  some $C_0 \ge 1$ and $\theta_0<1$,\footnote{Our default rule is to use
the same symbol for corresponding objects for $f, F$ and $\barF$ when
no ambiguity can arise given context. Thus $R$ is the name of the return time
function on $\Lambda, \Delta_0$ and $\bDelta_0$.}
 we fix $\beta$ with $1>\beta > \max \{ \theta_0, \sqrt{\alpha} \}$ where
 $\alpha$ is as in Lemma~\ref{lemma:jacobian}(3), and define
 a symbolic metric on $\bDelta$
 by $d_\beta(x,y) = \beta^{s(x,y)}$.
Since $\beta > \sqrt{\alpha}$, Lemma~\ref{lemma:jacobian}(3) implies that
$J\barF$ is log-Lipshitz with respect to this metric.
Let ${\cal B} = \{ \psi \in L^1(\bDelta, \bm): \|\psi\| < \infty \}$
where $\|\psi\| = \|\psi\|_\infty + \|\psi\|_{\lip}$ and
\[
\|\psi\|_\infty = \sup_{\ell, j} \sup_{x \in \bDelta_{\ell,j}}
                      |\psi(x)| \beta^\ell ,  \qquad
\|\psi\|_{\lip} = \sup_{\ell,j} \mbox{Lip}(\psi |_{\bDelta_{\ell,j}}) \beta^\ell \ .
\]
Lip$(\cdot)$ in the last displayed formula
is with respect to the symbolic metric $d_\beta$, and
$({\cal B}, \| \cdot\|)$ so defined is a Banach space.

Now consider the open system $(\barF, \bDelta; \bH)$ where $\bH$ is
a Markov hole. Following [BDM],
 we let $\bLp$ denote the transfer operator associated with
$\barF|_{\bDelta^1}$ defined on $\B$, i.e., for $\psi \in \B$
and $x \in \bDelta$,
\[
\bLp \psi(x) = 1_{\bDelta^0}(x)
\sum_{y \in \bDelta^0 \cap \barF^{-1}\{x\}} \psi (y) (J\barF(y))^{-1} .
\]
We say $(\barF, \bDelta; \bH)$ has a {\it spectral gap} if
\begin{itemize} \vspace{-3 pt}
 \item [(i)] $\bLp$ is quasi-compact
with a unique eigenvalue $\ra$ of maximum modulus, and
\vspace{-6 pt}
 \item [(ii)] $\ra$ is real with $\beta < \ra < 1$; it is
simple, with a one-dimensional eigenspace. \vspace{-3 pt}
\end{itemize}
Notice that if $h\in\B$ satisfies $\bLp h = \ra h$, then $h \bm$
defines a conditionally invariant measure for $\barF$ with eigenvalue $\ra$,
i.e. $\barF_*(h\bm)|_{\bar \Delta \setminus \bar H} = \ra \cdot h\bm$.

Finally, if $(F, \Delta)$ is an abstract hyperbolic tower
that projects onto $(\barF, \bDelta)$, and $\tilde H \subset \Delta$ is
a Markov hole which projects onto $\bH$, then
we say  $(F, \Delta; \tilde H)$
has a spectral gap if $(\barF, \bDelta; \bH)$ does.

\medskip
The conditions {\bf (A.1)}--{\bf (A.4)} in Sect.~\ref{tower results}
have now been made precise.

\subsection{Variational principles for $(\barF, \bDelta; \bH)$ and
$(F, \Delta; \pi^{-1}H)$}

\label{tower theorems}

As noted earlier, our aim in this section is
to prove, as an intermediate step for Theorems D and E,
a version of the corresponding results for the open system
$(F, \Delta; \pi^{-1}H)$. These results are
deduced from some previously known results for $(\barF, \bDelta; \bH)$,
which we first recall.

\bigskip
\noindent {\bf I. Results for expanding towers}

\medskip
We consider here an abstract expanding tower $(\barF, \bDelta; \bH)$ with
Markov holes. The following notation is used:
Let $\B$ be the function space above,
and define $\B_0$ to be the set of bounded functions in $\B$ whose Lipschitz
constant is also bounded, i.e.\ the definition of $\B_0$ is the same as that
of $\B$, but with the weights $\beta^\ell$ removed.
Let $\M_{\barF}(\bDelta^\infty)$ denote the set of invariant measures
on $\bDelta^\infty$,
and define
$$\G_{\bDelta} = \{ \bareta \in \M_{\barF}(\bDelta^\infty) \mid
\bareta(\log J \barF) < \infty \}\ .$$

\smallskip
\begin{theorem} {\rm (mostly \cite{bdm}; see Remark below)}
\label{thm:exp conv}
Assume $\bm \{R>n\} < C_0 \theta_0^n$, and
$(\barF, \bDelta; \bH)$ has a spectral gap with largest eigenvalue $\ra$.
Let $h_* \in \B$ be the unique eigenfunction of $\ra$ with $\int h_* d\bm=1$.
Then:
\begin{enumerate}
\item[(a)] There exist constants
$D>0$ and $\tau<1$ such that for all $\psi \in \B$,
\[
\| \ra^{-n}\bLp^n\psi - d(\psi)h_*\| \leq D\|\psi \| \tau^n, \; \; \;
\mbox{where }
d(\psi) = \lim_{n\to \infty} \ra^{-n} \int_{\bDelta^n} \psi \, d\bm < \infty .
\]
\end{enumerate}
Assume additionally

{\rm (*)}:  $\exists \ \bar C>0$ and
$\bar \theta \in (\ra^{-1} \theta_0, 1)$
such that $\log J\barF^n|_{\bDelta_0 \cap \{R=n\}} \le
\bar C \bar \theta^{-n}$ for all $n \ge 0$.
\begin{enumerate}
\item[(b)] $ \displaystyle
\log \ra = \pa_{\G_{\bDelta}} := \sup_{\bareta \in \G_{\bDelta}} \left\{ h_{\bareta}(\barF) - \int_{\bDelta} \log J\barF d\bareta \right\}\ . $
\item[(c)] Let $\bnu$ be defined by \[
\bnu(\varphi) = \lim_{n\to \infty} \ra^{-n} \int_{\bDelta^n} \varphi \, h_* \, d\bm
\qquad \mbox{for all $\varphi \in \B_0$ } .
\]
Then $\bnu \in \G_{\bDelta}$ and attains the supremum in (b).
\item[(d)] Other properties of $\bnu$ are that $(\barF, \bnu)$  is ergodic,
and  enjoys exponential decay of correlations between
$\vf$ and $\psi \circ \barF^n$ for  $\vf \in \B_0$ and $\psi \in L^\infty$.
\end{enumerate}
\end{theorem}

\medskip
\noindent {\it Remark.} The restriction $\bareta(\log J \barF) < \infty$, which appears in the
definition of $\G_{\bDelta}$, is omitted in [BDM], as is the condition (*),
which is extremely mild,\footnote{We observe that (*)
holds  for all the towers constructed in [BDM]; indeed, in that setting,
$\log J\barF^n|_{\bDelta_0 \cap \{ R = n \}} \le C n$ and all measures
$\bareta \in \M_{\barF}(\bDelta^\infty)$ satisfy $\bareta(\log J\barF)<\infty$.}
but a condition of this type is needed to ensure that
$\bnu \in \G_{\bDelta}$.
Since a main novelty of Theorem 3 is
the noncompactness of the phase space $\bDelta$, and
these conditions are directly connected to the finiteness of various
quantities, we will provide
sketches of corrected proofs of Theorem 3(b) and (c)
 in Sect. 5.3. The proofs of parts (a) and (d) in [BDM] are
unaffected.

\vskip .2in
\noindent {\bf II. Results for hyperbolic towers arising from $(f,M;H)$}

\medskip
We now return to the setting of Sect. 2.1, where $f:M\circlearrowleft$ is
a $C^{1+\epsilon}$ diffeomorphism with or without singularities. Let
$H \subset M$, and assume that the open system $(f,M;H)$
satisfies {\bf (A.1)}--{\bf (A.4)} in Sect. 2.2.

We first recall the following result proved in \cite{dwy} as part of our study of
billiard systems with holes.
Let $\tB$ be the class of measures $\sigma$ on $\Delta$ with
the following properties:
(i) $\sigma$ has absolutely continuous conditional measures
on unstable leaves; and
(ii) $\bpi_* \sigma = \bpsi_\sigma d\bm$ for some
$\bpsi_\sigma \in \B$.

\begin{theorem}[\cite{dwy}]
\label{thm:hyp conv}
Under the conditions above, the following hold for $(F, \Delta; \pi^{-1}H)$:

\begin{itemize}
\item[(a)] For all $\sigma \in \tB$ with $d(\bpsi_\sigma) > 0$, where $d(\bpsi_\sigma)>0$ is as in
Theorem~3(a),
\[
\log \ra = \lim_{n\to \infty} \frac{1}{n} \log \sigma (\Delta^n) \quad i.e. \quad
\rho(\sigma)= \log \ra\ .
\]
\item[(b)] There exists a conditionally
invariant distribution $\tmu_* \in \tB$, such that $\F_* \tmu_* = \ra \, \tmu_*$,
$\bpi_* \tmu_* = h_* \bm$, and for which the following hold:
For all $\sigma \in \tB$,
\[
\lim_{n \to \infty} \ra^{-n} \F_*^n\sigma = d(\bpsi_\sigma) \cdot \tmu_*\, ,
\; \; \;  \mbox{and if $d(\bpsi_\sigma)>0$, then } \; \; \;
\lim_{n \to \infty} \frac{\F_*^n\sigma}{\F_*^n\sigma(\Delta)} = \tmu_*
\]
where the convergence  is in the weak* topology.
\end{itemize}
\end{theorem}

The measure $\tmu_*$ can be thought of as the {\it physical measure
for the leaky system} $(F,\Delta; \pi^{-1}H)$.

We formulate in Theorem 5 the results which, along with Theorem 4,
will give the analogs of
Theorems D and E for $(F, \Delta; \pi^{-1}H)$. Let $\M_F(\Delta^\infty)$
denote the set of invariant probability measures supported on $\Delta^\infty$,
and define
\[
\G_\Delta = \{ \eta \in \M_F(\Delta^\infty) \mid \eta(\log J^u_\mu F) < \infty \}\ .
\]
Furthermore, let $C^0_b(\Delta)$ be the set of bounded functions on
$\Delta$ which are continuous on each $\dlj$. We postpone the definitions
of Lip$^s(\Delta)$ and Lip$^u(\Delta)$ (other function spaces that will appear)
until after the theorem.

\begin{theorem}
\label{thm:variational}
Let $(F, \Delta; \pi^{-1}H)$ be as above.  Then the following hold.
\begin{enumerate}
	\item[(a)]  $ \displaystyle
\log \ra = \pa_{\G_{\Delta}} = \sup_{\eta \in \G_{\Delta}} \left\{ h_\eta(F) - \int_\Delta \log J^u_\mu F d\eta \right\} .$
	\item[(b)] Let $\tnu$ be defined by
\[
\tnu(\varphi) = \lim_{n\to \infty} \ra^{-n} \int_{\Delta^n} \varphi \, d\tmu_*
\qquad \mbox{for all $\varphi \in C^0_b(\Delta)$ } .
\]
Then $\tnu \in \G_\Delta$ and it attains the supremum in (a).
\item[(c)] Other properties of $\tnu$ are that  $(F, \tnu)$ 	is ergodic, and
exhibits exponential decay of correlations between
$\vf$ and $\psi \circ F^n$ for $\vf \in \mbox{Lip}^u(\Delta)$
and $\psi \in \mbox{Lip}^s(\Delta)$.
\end{enumerate}
\end{theorem}

The function spaces $\mbox{Lip}^s(\Delta)$ and $\mbox{Lip}^u(\Delta)$
are defined as follows.
For $\omega^s \in \Gamma^s(\Delta)$ and $x,y \in \omega^s \subset \Delta_0$,
we denote by $d_s(x, y)$ the distance between $\pi(x)$ and $\pi(y)$ according to
the Riemannian metric on $M$, and extend $d_s$ to $\omega^s
\in \cup_{\ell >0} \Gamma^s(\Delta_\ell)$
by setting $d_s(F^\ell x, F^\ell y) = \alpha^\ell d_s(x,y)$ for all $\ell < R(x)$
and $y \in \omega^s(x)$. It then follows from Sect.~\ref{tower review}.I that
$d_s(F^nx, F^ny) \le C \alpha^n$ for all $n \geq 0$ and
$x,y \in \Delta$, $y \in \omega^s(x)$. For $\vf \in C^0_b$,
let $|\vf|^s_{\lip}$ be the supremum of Lipshitz constants of
$\vf|_{\omega^s}$ with respect to $d_s$, as $\omega^s$
ranges over all stable sets in $\Gamma^s(\Delta)$.
Then
Lip$^s(\Delta) = \{ \vf \in C^0_b : |\vf|^s_{\lip}  < \infty \}$. The function space
Lip$^u(\Delta)$ is defined similarly using $|\vf|^u_{\lip}$ where $|\vf|^u_{\lip}$
is the
Lipschitz constant of $\vf$ restricted to unstable
sets in the metric $d_\beta$.

\subsection{Outline of Proof of Theorem~\ref{thm:exp conv}(b),(c):
 \cite{bdm} amended}

\label{bdm proof}

We assume part (a) of Theorem 3 has been proved,
and proceed to the proofs of parts (b) and (c), following
mostly \cite{bdm} and highlighting several finiteness issues.

\bigskip
\noindent {\it 1. Return map to $\bDelta^\infty_0$ and the full shift \
$T: \Sigma_\infty \circlearrowleft$}

\smallskip
Since $\barF$ is not defined everywhere on $\bDelta$, let us first make
precise the definition of the survivor set $\bDelta^\infty$. Recall from Sect. 5.1 that
modulo a set of $\bm$-measure 0, $\bDelta_0$ is the disjoint union of
a countable number of closed subsets $(\bDelta_0)_j$ with the property that

(i) {\it in the absence of $\bH$}, $\barF^R$ maps each $(\bDelta_0)_j$
homeomorphically onto $\bDelta_0$, and

(ii) {\it with $\bH$ present}, each $(\bDelta_0)_j$ either falls entirely into
$\bH$ on its way up the tower or

\quad returns to $\bDelta_0$ intact.

\noindent We rename the subcollection $\{(\bDelta_0)_j\}$ that return to
$\bDelta_0$ in (ii) as $\{A_i\}$, and define
$$
\bDelta_0^\infty:= \bDelta^\infty \cap \bDelta_0 = \cap_{n \ge 0}
(\barF^R)^{-n} (\cup_i A_i)\ .
$$
It is easy to see that there is a bijection $\pi_0 :\bDelta_0^\infty \to
\Sigma_\infty = \Pi_{i=1}^\infty \{1,2,3, \cdots\}$ such that $\pi_0 \circ \barF^R
= T \circ \pi_0$ where $T: \Sigma_\infty \circlearrowleft$ the full shift.
Moreover, with $\bDelta_0^\infty$ given its relative
topology as a subset of $\bDelta_0$, and $ \Sigma_\infty$ given the topology
defined by cylinder sets, $\pi_0$ is a conjugating homeomorphism.

Let $\Z_n$ denote the set of cylinders in $\Sigma_\infty$
defined by coordinates $1, \cdots, n$, and write $\Z=\Z_1$.
We introduce a metric $\hat d$ on $\Sigma_\infty$ compatible with
its topology defined by $\{\Z_n\}$:
For $x,y \in \Sigma_\infty$, define
$\hat s(x,y) = \min \{ i \in \N \mid T^ix, T^iy$  lie in different $Z \in \Z \}$,
and let $\hat d(x,y) = \beta^{\hat s(x,y)}$ (where $\beta$ is as in Sect. 5.1).
We say a function $\phi: \Sigma_\infty \to \R$ is locally H\"older continuous if
\[
\sup_{Z \in \Z} \{ |\phi(x) - \phi(y) | \cdot \beta^{-\hat s(x,y)} : x, y \in Z \} < \infty  .
\]

\medskip
\noindent {\it 2. Sarig's abstract results on the pressure of \
$T: \Sigma_\infty \circlearrowleft$ }

\smallskip
We recall here a few relevant results for $T: \Sigma_\infty \circlearrowleft$.
These results were proved in \cite{sarig} in more general
settings of topologically mixing countable Markov shifts.
Given $\phi: \Sigma_\infty \to \R$, let
$S_n\phi = \sum_{i=0}^{n-1} \phi \circ T^i$. The
{\it Gurevic pressure} of $\phi$ is defined to be
\[
P_G(\phi, Z) = \lim_{n \to \infty} \frac 1n \log \left( \sum_{T^nx = x; \, x \in Z } e^{S_n\phi(x)} \right)
\]
where $Z$ is any fixed element of $\Z$. For $\phi$ locally H\"older continuous,
it is shown in \cite{sarig}, Theorem 1, that the limit above exists and is independent
of $Z$. This number is $\le \infty$ in general, and is equal to $\infty$ for many
$\phi$ given that $T$ is an infinite shift.

We will also need the following definitions:
The transfer operator associated with $\phi$ is given by
\[
\Lp_\phi \psi(x) = \sum_{Ty=x} e^{\phi(y)} \psi(y), \qquad
\mbox{for \ bounded \ $\psi $}.
\]
Let $\M_T(\Sigma_\infty)$ be the set of $T$-invariant Borel probability measures
on $\Sigma_\infty$. Given a potential $\phi: \Sigma_\infty \to \R$,
we say $\eta \in \M_T(\Sigma_\infty)$
is a {\it Gibbs measure} for $\phi$ if there exist constants $C>1$ and $P_\eta \in \R$
such that for any $n \ge 1$, $Z_n \in \Z_n$ and $x \in Z_n$,
\begin{equation}
\label{eq:gibbs}
C^{-1} e^{S_n\phi(x) - nP_\eta} \le \eta(Z_n) \leq Ce^{S_n\phi(x) - nP_\eta}.
\end{equation}

The following version of results from \cite{sarig}
are adequate for our purposes:

\begin{theorem} Let $T: \Sigma_\infty \circlearrowleft$ be as above, and let
$\phi: \Sigma_\infty \to \mathbb R$ be locally H\"older continuous.
Assume  $|\Lp_\phi 1|_\infty < \infty$. Then:
\begin{itemize} \vspace{-6 pt}
\item[(1)] \cite[Theorem 1]{sarig} $P_G(\phi) < \infty$;
\vspace{-6 pt}
\item[(2)] \cite[Theorem 3]{sarig} \vspace{-9 pt}
$$P_G(\phi) = \sup \{ h_\eta(T) + \int \phi d\eta \mid \eta \in \M_T(\Sigma_\infty)
\mbox{ and } \eta(-\phi) < \infty \}\ .
$$
\item[(3)] \cite[Theorem 8]{sarig} Suppose $\eta$ is a Gibbs measure for $\phi$, and
$\eta(-\phi)< \infty$. Then \vspace{-9 pt}
$$P_G(\phi)= P_\eta = h_\eta (T) + \int \phi d\eta .
$$
\end{itemize}
\end{theorem}

It follows from (1) and (2) above that for
$\eta \in \M_T(\Sigma_\infty)$, $h_\eta(T)<\infty$ provided
$|\Lp_\phi 1|_\infty < \infty$ and
$\eta(-\phi)<\infty$.

\medskip
\noindent {\it Notation:}
In what follows, we will identify $\barF^R: \bDelta^\infty_0 \circlearrowleft$ with
$T: \Sigma_\infty \circlearrowleft$ and use the two sets of notation
interchangeably. We also introduce the following notation: given
$\bareta \in \M_{\barF}(\bDelta^\infty)$, let $\bareta_0$ denote the measure
$\left(\frac{1}{\bareta(\bDelta_0^\infty)} \bareta \right) \!
|_{\bDelta_0^\infty}$. It is easy to see that $\bareta_0 \in
\M_{\barF^R}(\bDelta^\infty_0)$.

\bigskip
\noindent {\it 3. Relating pressure on $(\barF, \bDelta^\infty)$ to that on
$(\barF^R, \bDelta^\infty_0)$}

\smallskip
Let $\phi = - \log (\ra^R J\barF^R)$. The aim of this step is to prove that for every $\bareta \in \M_{\barF}(\bDelta^\infty)$ with
$\bareta(\log J\barF)< \infty$,
\begin{equation} \label{ineq}
\bareta(\bDelta_0^\infty)^{-1} \left\{h_{\bareta}(\barF) - \bareta(\log J\barF)
 - \log \ra \right\} = h_{\bareta_0}(\barF^R) + \bareta_0(\phi)
\le P_G(\phi) < \infty\ .
\end{equation}

The last two inequalities follow from
Theorem 6(a),(b) once we check
(i) $\phi$ is locally H\"older continuous
with respect to the metric $\hat d$, (ii) $|\Lp_\phi 1|_\infty < \infty$, and
(iii) $\bareta_0 (-\phi)< \infty$.

For (i), notice that by Lemma~\ref{lemma:jacobian}, $\phi$ is locally
H\"older continuous with respect to the separation time
metric $d_\beta$, and $\hat s(x,y) \le s(x,y)$.

For (ii), let
$Z(y)$ denote the element of $\Z$ containing $y \in \bDelta_0^\infty$.
We fix $x \in \bDelta_0^\infty$ and use the
bounded distortion of $J\barF^R$ given by Lemma~\ref{lemma:jacobian}(3)
to write
\[
\begin{split}
\Lp_\phi 1(x) & = \sum_{Ty=x} \ra^{-R(y)} (J\barF^R(y))^{-1}
\leq C \sum_{Ty=x} \ra^{-R(y)} \bm(Z(y)) \\
& \le C \sum_{n \geq 1} \ra^{-n} \bm\{R=n\} \le C' \sum_{n \ge 1} \ra^{-n} \theta_0^n < \infty\ .
\end{split}
\]
Here we have used $\ra> \theta_0$ and the fact that $\barF^R$ maps
each $Z \in \Z$ bijectively onto $\bDelta^\infty_0$.

For (iii), we will show $\bareta(\log J\barF)< \infty$ implies
$\bareta_0 (-\phi)< \infty$: Since  $J\barF \equiv 1$ on
$\bDelta \setminus \barF^{-1}\bDelta_0$,
\[
\int_{\bDelta_0^\infty} \log J\barF^R \, d\bareta_0 = \bareta(\bDelta_0^\infty)^{-1} \int_{\barF^{-1}\bDelta_0^\infty} \log J\barF \,d\bareta
= \bareta(\bDelta_0^\infty)^{-1} \int_{\bDelta^\infty} \log J\barF \,d\bareta\ .
\]
Thus if $\bareta(\log J\barF)< \infty$, then, noting
$\bareta(\bDelta_0^\infty) \int R d\bareta_0 = 1$, we have
\begin{equation}
\label{eq:finite}
\begin{split}
\bareta_0(-\phi) & =
\int_{\bDelta_0^\infty} \log (\ra^R J\barF^R) \, d\bareta_0 = \int_{\bDelta_0^\infty}
R \log \ra \, d\bareta_0
+ \int_{\bDelta_0^\infty} \log J\barF^R \, d\bareta_0 \\
& = \big(\log \ra + \bareta(\log J\barF)\big) \cdot
\bareta(\bDelta_0^\infty)^{-1} < \infty\ .
\end{split}
\end{equation}
This completes the verification of (i)--(iii).

The equality in (\ref{ineq}) follows from (\ref{eq:finite}) together with
the general formula of Abramov \cite{abramov}, which says that
$h_{\bareta}(\barF) = h_{\bareta_0}(\barF^R) \bareta(\bDelta^\infty_0)$.
In all the references we know of (e.g. \cite[\S 6.1]{petersen}), this equality
is proved assuming the invertibility of the transformation.
In the situation above, $\barF$ is clearly not invertible, but the same result
is easily deduced by passing to natural extensions; see Appendix B.

\bigskip
\noindent {\it 4. Existence of a pressure-maximizing invariant measure $\bnu$ }

\smallskip
Let $\bnu$ be the linear functional on
$C_b^0(\bDelta)$ defined by
\[
\bnu(\psi) = \lim_{n\to \infty} \ra^{-n}
            \int_{\bDelta} \bLp^n(h_* \psi)  \, d\bm
        = \lim_{n\to \infty} \ra^{-n} \int_{\bDelta^n} \psi h_* \, d\bm\ .
\]
We refer the reader to [BDM] for verification that $\bnu$ is a well defined,
$\barF$-invariant probability measure on $\bDelta^\infty$.

The aim of this step is to show that plugging $\bareta = \bnu$
into (\ref{ineq}), we get
\begin{equation} \label{eq}
h_{\bnu_0}(\barF^R) + \bnu_0(\phi) = P_G(\phi) = 0 \qquad {\rm and} \qquad
h_{\bnu}(\barF) - \bnu(\log J\barF) = \log \ra\ .
\end{equation}

Observe from the definition of $\phi$ in Step 3 that $e^{S_n\phi(x)} = \ra^{-S_nR(x)} (J(\barF^R)^n(x))^{-1}$.
The following lemma shows that $\bnu_0$
is a Gibbs measure for the potential $\phi$,
with $P_{\bnu_0}=0$.

\begin{lemma} \cite[Lemma 5.3]{bdm}
\label{lem:nu bounds}
There exists a constant $C>1$ such that for any $n\ge 1$, any $n$-cylinder
$Z_{n} \in \Z_{n}$, and any $y_* \in Z_{n}$,
\[
C^{-1} \ra^{-S_nR(y_*)} (J(\barF^R)^n(y_*))^{-1} \leq \bnu_0(Z_{n})
\leq C \ra^{-S_nR_n(y_*)} (J(\barF^R)^n(y_*))^{-1}\ .
\]
\end{lemma}

It remains only to check that $\bnu(\log J \barF) < \infty$, for
this bound implies $\bnu_0(-\phi)< \infty$ (see Step 3 above), and
once we have that, Theorem 6(c) gives the first equation in (\ref{eq}).
The second equation follows from (\ref{ineq}) and the first.

In what follows, $C$ will be used as a generic constant the value of which
is permitted to vary from line to line. To prove $\bnu(\log J \barF) < \infty$, we first estimate
\begin{equation}
\label{eq:nu decay}
\begin{split}
\bnu_0 \{R = n\} \; & =  \sum_{ Z \in \Z : R(Z) = n} \bnu_0(Z)
       \; \leq \; \sum_{Z \in \Z : R(Z) = n}
           C \ra^{-n} (J\barF^R(y_*))^{-1}  \\
    & \leq \; C \sum_{ Z \in \Z : R(Z) = n } \ra^{-n} \bm(Z)
    \; \leq \; C\theta_0^n \ra^{-n}\ ,
\end{split}
\end{equation}
where $y_*$ is an arbitrary point in $Z$. The first inequality comes from
Lemma 5.2, the second from Lemma 5.1(3), and the third from the
tail bound for $(\bDelta, \barF)$. Using the invariance of $\bnu$
and the fact that $J\barF \equiv 1$ on $\bDelta \setminus \barF^{-1}(\bDelta)$,
we obtain $\bnu(\log (J\barF))$
$$
= \sum_{n \ge 1} \sum_{R(Z) = n}  \int_Z  \log (J\barF^n) d \bnu
= \sum_{n \ge 1} \bnu\{R=n\} |\log J\barF^n|_\infty
\le C \sum_{n \ge 1} (\theta_0 \ra^{-1})^n \bar \theta^{-n} < \infty\ .
$$
The inequalities above  come from condition (*) in Theorem 3; this is
the only place in the entire proof that uses this condition. We have also used
the fact that $\bnu\{R=n\}$ is bounded by $\bnu(\bDelta^\infty_0)$ times
the last quantity in (\ref{eq:nu decay}).

\bigskip
Parts (b) and (c) of Theorem 3  follow immediately from Steps 3 and 4.
\hfill $\square$

\subsection{Proof of Theorem~\ref{thm:variational}}

\label{tower proofs}

We will prove this theorem by leveraging the corresponding results for
expanding towers.

\bigskip
\noindent {\it Variational principle} (Theorem 5(a),(b)): First, we show
\begin{equation} \label{comp}
\sup_{\eta \in \G_\Delta} \left\{h_\eta(F)-\int \log J^u_\mu(F) d\eta \right\}
\le \sup_{\bareta \in \G_{\bDelta}} \left\{h_{\bareta}(\barF)-
\int \log J(\barF) d\bareta \right\}\ ,
\end{equation}
which follows immediately from the following lemma:

\begin{lemma}
\label{lem:var tnu}
Let $\eta \in \G_\Delta$ and define $\bareta = \bpi_* \eta$.
Then $\bareta \in \G_{\bDelta}$ and
\begin{enumerate}  \vspace{-6 pt}
  \item[(i)]  $\int_\Delta \log J^u_\mu F \, d\eta = \int_{\bDelta} \log J\barF \, d\bareta$;
                   \vspace{-6 pt}
  \item[(ii)]  $h_\eta(F) = h_{\bareta}(\barF)$.
\end{enumerate}
\end{lemma}

\begin{proof}[Proof of Lemma~\ref{lem:var tnu}]
Let $\eta \in \G_\Delta$.
The fact that $\bareta = \bpi_* \eta \in {\cal M}_{\barF}(\bDelta^\infty)$
is clear. That $\bareta \in \G_{\bDelta}$ will follow once we prove
 assertion (i) of the lemma:
From Sect. 5.1I, we see that $\log J^u_\mu F$ and $J\barF$ are related by
$J\barF \circ \bpi = J^u_\mu F \cdot e^{a \circ F - a}$ for a bounded
function $a$ (Lemma 5.1(4)). It follows
that
\begin{equation} \label{comp2}
\int_\Delta \log J^u_\mu F \, d\eta = \int_\Delta \left(\log J\barF \circ \bpi
+ a - a \circ F\right) \, d \eta
= \int_{\bDelta} \log J\barF \, d\bareta\ ,
\end{equation}
the invariance of $\eta$ being used in the second equality.

Assertion (ii) follows from (a) the entropy of a transformation is
equal to that of its natural extension, and (b) the natural extension of
$(F, \eta)$ is isomorphic to that of $(\barF, \bareta)$.
See Appendix B for more detail on (b).
\end{proof}

To complete the proof,
we will show that
(i) the results of Theorem 3 are applicable to the quotient tower, and
(ii) $\tnu$ as defined in
 part (b) is in $\G_\Delta$ and projects to $\bnu$.
These two steps together will show that (\ref{comp})
is in fact an equality, and the quantity on the right is $=\log \ra$.

 To apply Theorem 3, it suffices to show
that condition (*) holds in the present setting, i.e. for the quotient tower
of a hyperbolic tower arising from $(f,M;H)$ and
 satisfying {\bf (A.1)}--{\bf (A.4)}.
Notice first that (*) holds if $\|Df\|$ is bounded, for
$J^u_\mu F^n$ can grow at most exponentially and
$J\barF^n$ on the corresponding set is $ \le J^u_\mu F^n \cdot e^{|a|_\infty}$ where $a$ is
as in Lemma 5.3. Thus there is a potential
problem only in the setting of Theorem C, where $\|Df\|$ may become
arbitrarily large as one approaches the singularity set $\Si$. Here
it is {\bf (H.2)} of Section~\ref{tower review}.II and \eqref{eq:d1 blowup}
in Sect. 2.1.III that give what we need: Since $F^{-j} \Delta_j =
\{x \in \Lambda: R(x)>j \}$, we have $d(\pi \Delta_j, \Si) \geq \delta \xi_1^{-j}$ for some $\delta>0$ and $\xi_1>1$ by {\bf (H.2)}. This together with
\eqref{eq:d1 blowup} implies that on $\pi(\Delta_j)$,
 $ |\det (Df|_{E^u})| \le (C_1 \delta^{-a} \xi_1^{aj})^p$ where
$p$ is the dimension of $E^u$. Thus
on $\Delta_0 \cap \{R=n\}$, we have
$$
\log J^u_\mu F^n
=  \sum_{j=0}^{n-1}  \log |\det (Df|_{E^u}) \circ f^j|
\le \ {\rm const} \ n^2\ ,
$$
which, as explained above, gives (*).

It remains to produce $\tnu$ with the properties in (ii).
Let $\tmu_*$ be the physical conditionally invariant distribution from
Theorem~\ref{thm:hyp conv}.
For $\vf \in \mbox{Lip}^u(\Delta)$, define $\tmu^\vf$ to be the measure such that
$d\tmu^\vf = \vf d\tmu_*$.  Notice that since $\bpi_* \tmu_* \in \B$ and
$|\vf|_\infty + |\vf|^u_{\lip} < \infty$, we also have $\bpi_* \tmu^\vf \in \B$.
Let $\bpsi_\vf$ denote the density of $\bpi_* \tmu^\vf$ with respect to $\bm$.
Now using Theorem~\ref{thm:hyp conv}(b),
\[
\lim_{n \to \infty} \ra^{-n} \int_{\Delta^n} \vf \, d\tmu_*
= \lim_{n \to \infty} \ra^{-n} \tmu^\vf(1_{\Delta^n})
= \lim_{n \to \infty} \ra^{-n} \F_*^n\tmu^\vf(1)
= d(\bpsi_\vf) .
\]
Let $Q(\vf) = d(\bpsi_\vf)$.  Then
$Q$ is clearly linear in $\vf$, positive and satisfies
$Q(1)=1$.  Also, $|Q(\vf)| \leq |\vf|_\infty Q(1)$ so that $Q$ extends to a bounded linear
functional on $C^0_b(\Delta)$.  By the Riesz representation theorem, there exists a
unique Borel probability measure $\tnu$ satisfying $\tnu(\vf) = Q(\vf)$ for each
$\vf \in C^0_b(\Delta)$.   Since $1_{\Delta^n} = 1_{\Delta^{n-1}} \circ \F$,
the invariance of $\tnu$ follows from
\[
\tnu(\vf \circ \F) = \lim_{n \to \infty} \ra^{-n} \tmu_* (\vf \circ \F \cdot 1_{\Delta^n})
= \lim_{n \to \infty} \ra^{-n}  \F_*\tmu_* (\vf \cdot 1_{\Delta^{n-1}})
= \lim_{n \to \infty} \ra^{1-n}  \tmu_*(\vf \cdot 1_{\Delta^{n-1}}) = \tnu(\vf)
\]
by the conditional invariance of $\tmu_*$.

Since $(\bpi_*\tmu_*)|_{\Delta^n} = (h_*\bm) |_{\bDelta^n}$ for every $n$,
it follows that $\bpi_* \tnu = \bnu$. To place $\tnu \in \G_\Delta$,
we need to show $\tnu(\log J^u_\mu F) < \infty$. This is true by
(\ref{comp2}) with $\eta = \tnu$ and the fact that the integral
on the right is known to be finite.

\bigskip
 \noindent
{\em Other properties of $\tnu$.}
The ergodicity of $\tnu$ follows from that
of $\bnu$. To show that
$\tnu$ enjoys exponential decay of correlations, we begin by decomposing
$\tnu$ into conditional measures $\tnu^s$ on $\omega^s$-leaves and a transverse measure
$\tnu_T$ on the set of stable leaves in each $\dlj$.  For
$\vf \in C^0_b$, define $\bvf(x) = \int_{\omega^s(x)} \vf \, d\tnu^s$.
Since each $\tnu^s$ is a probability measure, we have $\bvf \in C^0_b$.  By definition,
$\bvf$ is constant on $\omega^s$-leaves and $\tnu(\vf) = \tnu(\bvf) = \bnu(\bvf)$.  Also
if $\vf \in \mbox{Lip}^u(\Delta)$, then $\bvf \in \mbox{Lip}^u(\Delta)$ so that we may
consider $\bvf \in \B_0$ as a function on $\bDelta$.

Now let $\vf \in \mbox{Lip}^u(\Delta)$ and $\psi \in \mbox{Lip}^s(\Delta)$
with $\tnu(\vf) = \tnu(\psi) =0$.
Define $\bvf$ as above and let
$\bpsi_k(x) = \int_{\omega^s(x)} \psi \circ F^k \, d\tnu^s$.
Note that $\tnu(\bpsi_k) = \tnu(\psi) = 0$.
Then setting $n = k + \ell$, we write
\begin{equation}
\label{eq:corr}
\tnu(\vf \, \psi \circ F^n) = \tnu(\vf \, (\psi \circ F^n- \bpsi_k \circ F^\ell))
+ \tnu((\vf - \bvf) \, \bpsi_k \circ F^\ell) + \tnu(\bvf \, \bpsi_k \circ F^\ell).
\end{equation}
Since $\bvf$ and $\bpsi_k$ are constant on $\omega^s$-leaves, we have
$\tnu(\bvf \, \bpsi_k \circ F^\ell) = \bnu(\bvf \, \bpsi_k \circ \barF^\ell)$ and
$\bnu(\bvf) = \bnu(\bpsi_k) =0$.   Then since $\bvf \in \B_0$ and $\bpsi_k \in L^\infty(\bDelta)$,
the last term in \eqref{eq:corr} is $\leq C \tau^\ell \| \bvf \|_{\B_0} |\bpsi|_\infty$
for some $\tau < 1$
by Theorem~\ref{thm:exp conv}(d) (see also \cite[Prop. 2.8]{bdm}).

The second term of \eqref{eq:corr} is identically 0 since,
\[
\begin{split}
\tnu(\vf \, \bpsi_k\circ F^\ell) & = \int_{\Gamma^s(\Delta)}
\Big( \int_{\omega^s} \vf \, \bpsi_k \circ F^\ell \, d\tnu^s \Big) \, d\tnu_T
= \int_{\Gamma^s(\Delta)} \Big( \int_{\omega^s} \vf \, d\tnu^s \Big) \,
\bpsi_k \circ F^\ell \, d\tnu_T  \\
& = \int_{\Gamma^s(\Delta)} \bvf \, \bpsi_k \circ F^\ell \, d\tnu_T =
\tnu(\bvf \, \bpsi_k \circ F^\ell).
\end{split}
\]
To estimate the first term in \eqref{eq:corr}, notice that $|\psi \circ F^n - \bpsi_k\circ F^\ell|_\infty
\le |\psi \circ F^k - \bpsi_k |_\infty$.  Then
since $\psi \circ F^k$ is continuous on each
$\omega^s$, there must exist $x,y \in \omega^s$ such that
$\psi \circ F^k(x) \leq \bpsi_k (\omega^s) \leq \psi \circ F^k(y)$.  Thus
\begin{equation}
\label{eq:stable contract}
|\tnu(\vf \, (\psi \circ F^n- \bpsi_k \circ F^\ell))| \leq |\vf|_\infty \,
|\psi \circ F^k - \bpsi_k|_\infty \leq 2 |\vf|_\infty \, |\psi|^s_{\lip} C \alpha^k
\end{equation}
by definition of $d_s$.  Taking both $k$ and $\ell$ to be approximately $n/2$ completes
the proof.


\section{Proof of Theorems D and E}
\label{proof of thm tower}

We now return to the original open system $(f,M; H)$, where $f$ is any
dynamical system admitting a tower with the properties in Sect.~\ref{tower results} (see Sect. 5.1 for detail).

\subsection{Proof of Theorem~D}
Let $\tmusrb$ be the SRB measure for $F$ on $\Delta$ before the removal
of the hole.
Note that $\pi_* \tmusrb = \musrb$, the unique SRB measure for $f$
with $\musrb(\Lambda)>0$.  It follows from
\cite[Section 2]{young tower} that $\tmusrb \in \tB$, so that
$\rho(\tmusrb) = \log \ra$ by Theorem~\ref{thm:hyp conv}(a).
Since $\musrb(M^n) = \pi_* \tmusrb(M^n) = \tmusrb(\Delta^n)$ for each $n \ge 0$,
we have $\rho(\musrb) = \log \ra$ and part (a) of Theorem~D is
proved.

To prove part (b),
define $\mu_* = \pi_* \tmu_*$ where $\tmu_*$ is the conditionally invariant measure
from Theorem~4.  We use $\f^n = f^n|_{M^n}$ to describe the surviving dynamics
at time $n$.   It follows from the
relation $\f \circ \pi = \pi \circ \F$ that for any Borel subset $A$ of $M\setminus H$,
we have
\begin{equation}
\label{eq:conditionally inv}
\mu_*(\f^{-1}A) = \tmu_*(\pi^{-1} (\f^{-1}A)) = \tmu_* (\F^{-1} (\pi^{-1}A)) = \ra \, \tmu_*(\pi^{-1}A) = \ra \, \mu_*(A)
\end{equation}
so that $\mu_*$ is a conditionally invariant measure for $\f$ with eigenvalue $\ra$.
By Theorem~4(b),
\[
\lim_{n\to \infty} \frac{\f^n_* \musrb}{\f^n_*\musrb(M)} =
\lim_{n \to \infty} \frac{\pi_*(\F^n_* \tmusrb)}{\F^n_*\tmusrb(\Delta)}
= \pi_* (\tmu_*) = \mu_*,
\]
proving part (b).

To prove part (c), define $\hat \nu = \pi_* \tnu$
where $\tnu$ is from Theorem~5.  Arguing
analogously to \eqref{eq:conditionally inv}, we see that
$\hat \nu$ is an invariant measure for $f$ supported on
$\pi(\Delta^\infty) \subseteq \Omega$.  Write
$J^uf(x) = |\det (Df_x|_{E^u(x)})|$.
We will show

(i) $\int_\Delta \log J^u_\mu Fd \tnu = \int_M \log J^uf \, d\hat \nu$,
and

(ii) $h_{\tnu}(F) = h_{\hat \nu}(f)$.

\noindent
Integrating over sets of the form $\cup_{i=0}^{n-1} F^i(\Delta_0
\cap \{R=n\})$ before
summing over $n$, we see that the left side of (i) is equal to
$\int_{\Delta_0} \log J^u_\mu F^R d\tnu$ and
the right side is equal to $\int_M \log J^uf^R \, d\pi_*(\tnu|_{\Delta_0})$,
the latter using the invariance of $\tnu$ and relation $\pi_* (F^i)_* =
(f^i)_* \pi_*$. These two integrals are easily seen to be equal:
Let $J^u\pi$ denote the Jacobian with respect to $\tmu_\omega$ for
$\omega \in \Gamma^u(\Delta)$ and $\mu_{\omega'}$ where $\pi(\omega)=
\omega'$. Then on $\Delta_0$, $J^u\pi \equiv 1$ as $\Delta_0$ is an isometric
copy of $\Lambda$, so we have
$$
J^uf^R \circ \pi = J^u_\mu F^R \cdot \frac{J^u \pi \circ F^R}{J^u \pi}
= J^u_\mu F^R\ .
$$

For (ii), that $h_{\hat \nu}(f) \leq h_{\tnu}(F)$ is obvious.  The reverse inequality follows from
\cite[Proposition 2.8]{buzzi} since $\pi$ is at most countable-to-one.   Combining
(i) and (ii) and using Theorem~5(b),
\[
\rho(\musrb) = \log \ra = h_{\tnu}(F) - \int_\Delta \log J^u_\mu F \, d\tnu
= h_{\hat \nu}(f) - \int_{M} \log J^uf \, d\hat \nu = P_{\hat \nu}  .
\]
The following lemma completes the proof of part (c).

\begin{lemma}
\label{lem:Gsrb}
$\hat \nu \in \G_H \cap \G_\Si$
\end{lemma}

\begin{proof}
That $\hat \nu$ is ergodic follows immediately from the fact that $\tnu$ is ergodic.
In order to show that $\hat \nu \in \G_H \cap \G_\Si$, we
will show that there exist $C, \alpha >0$ such that for each $\ve >0$,
$\hat \nu(N_\ve(\Si \cup \partial H)) \leq C\ve^\alpha$.
Once this is established, we conclude by an argument similar to
Lemma~\ref{lem:approach}
that $\hat \nu$-a.e.\ point approaches $\Si \cup \partial H$ at an arbitrarily slow
exponential rate.

To establish this bound, we need estimates on how $\tnu$ decays up the
levels of the tower.
Recall $\bnu = \bpi_*\tnu$.  In the proof
of Theorem~\ref{thm:variational}, we established that
$d(\pi \Delta_\ell, \Si \cup \partial H) \geq \delta \xi_1^{-\ell}$, $\ell \ge 0$,
by using {\bf (H.2)} of Section~\ref{tower review}.II. Thus we have
\[
\hat \nu(N_\ve(\Si \cup \partial H)) \leq \tnu \! \left( \cup_{\ell : \delta \xi_1^ {-\ell} \leq \ve}
\Delta_\ell \right) \leq \sum_{\ell \geq \log(\delta/\ve)/\log \xi_1} C'
\theta_0^\ell \ra^{-\ell} \leq C'' (\delta^{-1} \ve)^{\log (\ra\theta_0^
{-1})/\log \xi_1},
\]
using  \eqref{eq:nu decay} and $\tnu(\Delta_\ell) = \bnu(\bDelta_\ell)$.
\end{proof}

Finally, we prove part (d).
If $\vf$ is a continuous function on $M$, we define its lift to $\Delta$
by $\tvf = \vf \circ \pi$.  This lift is continuous on each $\dlj$ and
$|\tvf|_\infty \leq |\vf|_\infty$
so that $\tvf \in C^0_b(\Delta)$.  Using Theorem~5(b), we have
\[
\hat \nu(\vf) = \tnu(\tvf) = \lim_{n\to \infty} \ra^{-n} \int_{\Delta^n} \tvf \, d\tmu_*
= \lim_{n \to \infty} \ra^{-n} \int_{M^n} \vf \, d\mu_*   ,
\]
since $\mu_* = \pi_* \tmu_*$.

To complete the proof of Theorem D, it remains to show that $\hat \nu$
enjoys exponential decay of correlations.
Let $C^p(M)$ denote the H\"older continuous functions on $M$ with exponent $p$.
If $\vf \in C^p(M)$ and $p \geq \log \beta / \log \alpha$, then
$\vf \circ \pi \in \mbox{Lip}^u(\Delta)$.  This
can be proved as in \cite[Section 6]{demers norms}.
Also, taking $\psi \in C^p(M)$,  for $x \in \dlj$, $y \in \omega^s(x)$ and
$x_0 = F^{-\ell}x$, $y_0 = F^{-\ell}y$, we have
\[
\begin{split}
|\psi \circ \pi \circ F^n(x) - \psi \circ \pi \circ F^n(y) |
& \leq |\psi|_{C^p} d(\pi (F^{n+\ell}x_0), \pi(F^{n+\ell}y_0))^p  \\
& \leq |\psi|_{C^p} d(f^{n+\ell}(\pi x_0), f^{n+\ell}(\pi y_0))^p
\leq |\psi|_{C^p} C \alpha^{np} .
\end{split}
\]
So taking $\vf, \psi \in C^p(M)$, we may apply
\eqref{eq:stable contract} to $\psi \circ \pi$.  We follow \eqref{eq:corr} and note that
\[
\hat \nu(\vf \, \psi \circ f^n) = \tnu(\vf \circ \pi \cdot \psi \circ f^n \circ \pi)
= \tnu(\vf \circ \pi \cdot \psi \circ \pi \circ F^n) ,
\]
to conclude that the exponential decay of correlations for $\hat \nu$ follows
from that for $\tnu$.

\subsection{Proof of Theorem~E}

As an immediate corollary of Theorem D, we have
\begin{equation} \label{halfeq}
\rho(\musrb) \le \pa_{\G_H \cap \G_\Si}\ ,
\end{equation}
since we have identified a measure, namely $\hat \nu$, in $\G_H \cap \G_\Si$
with $P_{\hat \nu} = \rho(\musrb)$.
 We will call upon the results
in Sect. 2.1 to provide the reverse inequality -- once we put ourselves
in a viable setup.
Notice that $\hat \nu = \pi_* \tnu$ necessarily gives positive
measure to $\Lambda = \pi(\Delta_0)$.

\medskip
\noindent
\emph{(a) $\musrb = \vf \mu$ with $\vf \geq \delta>0$ on a neighborhood
of $\Lambda$.}   In this case,
$\hat \nu \in \G_\vf$ since we can simply take the
set $Z$ in the definition of
$\G_\vf$ to be this neighborhood.
By Theorem~C,
$\rho(\musrb) = \rho(\mu_\vf) \geq \pa_{\G_H \cup \G_\Si \cup \G_\vf}$.
This together with (\ref{halfeq}) gives the desired result.

\medskip
\noindent
\emph{(b) $\Lambda$ is contained in
a $\musrb$-hyperbolic product set.}  Taking this set to be $\Pi$ in the
definition of $\Gsrb$, it is immediate that $\hat \nu \in \Gsrb$.
Theorem~C and (\ref{halfeq}) then give the two halves of the desired equality.


\section*{Appendix}

\subsection*{A. Lyapunov charts for maps with singularities}
\label{charts}

In this section, we prove the statements $(a)(iii')$, $(b)(ii')$, and $(b)(iii')$ made in
Section~4.2.4 regarding the Lyapunov charts $\{ \Phi_x \}$.
All notation is as in Section~4.

We begin with $\nu \in \G_\Si$ and the set $V'$ of regular points in the sense
of Oseledec.  Each $x \in V'$ has  $p$ distinct Lyapunov exponents
$\lambda_1, \ldots, \lambda_p$ with corresponding subspaces
$E_1(x), \ldots, E_p(x)$ such that $T_xM = \oplus_i E_i(x)$.  Let
$g_\ve(x) = \frac 13 \min \{\ve, d(x, \Si) \}$.

Fix $\delta >0$.  It follows by standard arguments (see \cite[Sect.~3.1]{young notes})
that for $\nu$-typical $x$,
one can define an inner product, $\langle \cdot, \cdot \rangle'_x$,
on the tangent space $T_xM$ such that item $(b)(i)$
of Proposition~4.1 holds.  Denote by $\| \cdot \|'_x$ the norm
induced by $\langle \cdot, \cdot \rangle'_x$ and by $\| \cdot \|_x$ the Euclidean norm
on $T_xM$.  It follows from the
same construction that there exists a measurable function
$\ell_0(x): V' \to [1,\infty)$,
with $\ell_0(f^ix) < e^{2\delta i} \ell_0(x)$ for $i \ge 0$ and
\begin{equation}
\label{eq:new norm}
p^{-1/2} \|v\|_x \leq \|v\|'_x \leq \ell_0(x) \|v\|_x  \qquad \mbox{for all } v \in T_xM .
\end{equation}
Define a linear map $L_x:T_xM \to \mathbb{R}^d$
which takes $E_i(x)$ to $\{ 0 \} \times \cdots \times \mathbb{R}^{m_i(x)} \times \cdots \times \{ 0 \}$
for each $i$ and such that
$\langle L_xu, L_xv \rangle_x = \langle u,v \rangle'_x$.
Then $\Phi_x := \mbox{exp}_x \circ L_x^{-1}$ is a Lyapunov chart satisfying
properties $(a)(i)$ and $(a)(ii)$ of Proposition~4.1.

The construction outlined thus far is standard and is not affected by
the presence of singularities (see \cite[Part I, Theorem 2.2]{katok}).
We now proceed to prove the statements
of Section~4.2.4 which are affected by the singularities.  We drop the subscript
$x$ for simplicity of notation and write $\| \cdot \|$ and $\| \cdot \|'$ in what follows.

Notice that in the notation of Sect.~4, $R(r) = R(r; \| \cdot \|')$ denotes the
ball of radius $r$ in the (Lyapunov) norm $\| \cdot \|'$ since that is the norm
of the Lyapunov charts $\Phi_x$.  To distinguish between norms,
we use $R(r ; \| \cdot \|)$ to denote
the ball of radius $r$ in the Euclidean norm on $T_xM$.  We identify $T_xM$ and
$\R^d$ and view $L_x$ formally as a change of norm.

\medskip
\noindent
{\em Proof of $(a)(iii')$.}
Recall the injectivity radius from Section~\ref{upper bound},
$\iota(x,U) \geq \min \{ s, d(x, M\setminus U)^\varsigma \}$, given by equation \eqref{eq:exp}.  Since we have assumed $b \geq \varsigma$, we have
$\iota(x,U) \geq g_\ve(x)^b$ for $\ve \leq s$.  Thus again using \eqref{eq:exp},
for $y \in B(x, 0, g_\ve(x)^b)$ and $w = $exp$_x^{-1} y$, we have
\[
\|D(\mbox{exp}_x)(w)\| \leq c_0 \qquad \mbox{and} \qquad
\|D(\mbox{exp}_x^{-1})(w)\| \leq c_0.
\]
This implies that exp$_x$ maps $R(c_0^{-1} g_\ve(x)^b; \| \cdot \|)$ injectively into
$B(x,0,g_\ve(x)^b)$.  Thus for $u,v \in R(c_0^{-1} g_\ve(x)^b; \| \cdot \| )$, we
use \eqref{eq:new norm} to estimate,
\[
\begin{split}
d(\Phi_xu, \Phi_xv)  & \leq d(\mbox{exp}_x \circ L_x^{-1}u, \mbox{exp}_x \circ L_x^{-1}v)
\leq c_0 \| L_x^{-1}u - L_x^{-1}v\| \leq c_0 \sqrt{p} \, \|u-v\|' \\
&  \leq c_0 \sqrt{p} \, \ell_0(x) \|u-v\| \leq c_0^2 \sqrt{p} \, \ell_0(x) d(\Phi_xu,\Phi_xv) ,
\end{split}
\]
which
establishes $(a)(iii')$ with $K = c_0 \sqrt{p}$ and
$\ell_1(x) = c_0^2 \sqrt{p} \, \ell_0(x)$.
Note that by \eqref{eq:new norm},
$R(\ell_1^{-1}(x) g_\ve(x)^b ) \subseteq R(c_0^{-1} g_\ve(x)^b\, ; \| \cdot \|)$,
with room to spare.

\medskip
\noindent
{\em Proof of $(b)(iii')$.}
Recall that
\[
\hat f_x = \mbox{exp}^{-1}_{fx} \circ f \circ \mbox{exp}_x
\; \; \mbox{while} \; \;
\tf_x = \Phi_x^{-1} \circ f \circ \Phi_x = L_{fx} \circ \hat f_x \circ L_x^{-1} .
\]
Taking $u,v, h \in R(c_0^{-1} g_\ve(x)^b; \| \cdot \| )$,
we use \eqref{eq:new norm} to estimate
\begin{equation}
\label{eq:lip}
\frac{\| D\tf_x(u)h - D\tf_x(v)h\|'}{\|h\|'}
\leq \frac{\| D\hat f_x(u)h - D\hat f_x(v)h\|}{\|h\|} \sqrt{p} \, \ell_0(x)
\leq \| D^2\hat f_x(z) \| \|u-v\| \sqrt{p} \, \ell_0(x)
\end{equation}
for some $z \in R(c_0^{-1} g_\ve(x)^b; \| \cdot \|)$.   By \eqref{eq:d2 blowup},
$\| D^2\hat f_x(z) \| \leq C_1 d(\mbox{exp}_x(z), \Si)^{-b}$.   Since
exp$_x(z) \in B(x,0,g_\ve(x)^b)$, we have
$d(\mbox{exp}_x(z), \Si) \geq g_\ve(x)$,
so that $\| D^2\hat f_x(z) \| \leq C_1 g_\ve(x)^{-b}$.  Finally, since
$\|u-v\| \leq \sqrt{p} \, \|u-v\|'$, we conclude that
\[
\mbox{Lip}(D\tf_x) \leq p \, \ell_0(x) C_1 g_\ve(x)^{-b}.
\]
The statement follows by taking $\ell(x)$ to be the larger of
$p \, C_1 \ell_0(x)$ and $\ell_1(x) = c_0^2 \sqrt{p} \, \ell_0(x)$.

\medskip
\noindent
{\em Proof of $(b)(ii')$.}  We use \eqref{eq:lip} with $v=0$ and
$u \in R(\delta \ell(x)^{-1} g_\ve(x)^b)$.  This yields
\[
\| D\tf_x(u) - D\tf_x(0) \|' \le \ell(x) g_\ve(x)^{-b} \| u\|' \le \delta.
\]
This implies that restricted to $R(\delta \ell(x)^{-1} g_\ve(x)^b)$,
we have Lip$(\tf_x - D\tf_x(0)) \le \delta$ as required.


\subsection*{B. Natural extensions of tower maps}

Let $T: (X, \Sigma, \nu) \circlearrowleft$ be a measure-preserving transformation
(mpt)
of a probability space. Recall that the {\it natural extension} of $T: (X, \Sigma, \nu) \circlearrowleft$, denoted here by $T^\sharp: (X^\sharp, \Sigma^\sharp, \nu^\sharp) \circlearrowleft$, is defined as follows:
$$
X^\sharp = \{(x_1, x_2, \cdots) \in \Pi_{i=0}^\infty X : T(x_{i+1})=x_i\},
$$
$$T^\sharp(x_1, x_2, \cdots) = (T(x_1), x_1, x_2, \cdots),
$$
$\Sigma^\sharp$ is generated by cylinder sets with $\Sigma$ in each coordinate, and
$$
\nu^\sharp\{x_1 \in A_1, \cdots, x_n \in A_n\} = \nu(A_n \cap T^{-1}A_{n-1}
\cap \cdots \cap T^{-(n-1)}A_1)\ .
$$

These following facts about tower maps (see Sect. 5.1 for notation)
are used:

\bigskip
\noindent (1) Consider $F: (\Delta^\infty, \Sigma, \eta) \circlearrowleft$
where $\eta$ is any $F$-invariant Borel probability measure, and
let $\barF: (\bDelta^\infty, \overline{\Sigma}, \bareta) \circlearrowleft$ be the
corresponding quotient system. We claim that the natural extensions of
these two mpt's are isomorphic.

\medskip
\noindent {\it Proof.} Define
$\bpi^\sharp : \Delta^\sharp \to \bDelta^\sharp$ by
$\bpi^\sharp(x_1, x_2, \cdots) = (\bpi(x_1), \bpi(x_2), \cdots)$.
Clearly, $\bpi^\sharp \circ F^\sharp = \barF^\sharp \circ \bpi^\sharp$,
$\bpi^\sharp_*(\eta^\sharp)=\bareta^\sharp$, and $\bpi^\sharp$ is onto.
The assertion follows once we show $\bpi^\sharp$ is 1-1.

Suppose $\bpi^\sharp(x_1, x_2, \cdots) =
\bpi^\sharp(y_1, y_2, \cdots)$.  Letting $\omega^s(x_n)$
denote the stable set of $x_n$, we
have, by definition, $x_1 \in \cap_{n=1}^\infty F^{n-1}(\omega^s(x_n))$.
The uniform contraction of $F$ along stable sets implies that this intersection
consists of a single point. Likewise, $\{y_1\} = \cap_{n=1}^\infty
F^{n-1}(\omega^s(y_n))$. Since $\bpi(x_n)=\bpi(y_n)$ is equivalent to
$\omega^s(x_n)=\omega^s(y_n)$, we have proved $x_1=y_1$.
Applying the same argument to the sequences $(x_k, x_{k+1}, \ldots )$
and $(y_k, y_{k+1}, \ldots)$, we conclude that $x_k=y_k$ for all $k \geq 1$.
\hfill $\square$

\bigskip
\noindent (2) Next given $\barF: (\bDelta^\infty, \overline{\Sigma}, \bareta) \circlearrowleft$ and $\bDelta_0^\infty \subset \bDelta^\infty$, we call
$\barF^R: (\bDelta^\infty_0, \overline{\Sigma}_0, \bareta_0) \circlearrowleft$
with $\bareta_0 = \bareta|_{\bDelta_0}$ normalized its
{\it induced map} on $\bDelta_0$, and claim that the induced
map of $\barF^\sharp$ on
$\bDelta_0^\sharp = \{(x_1, x_2, \cdots) \in \bDelta^\sharp : x_1 \in \bDelta_0\}$
is the natural extension of $\barF^R$. The proof is easy.

\bigskip
\noindent {\bf Fact.} For an arbitrary mpt $T: (X, \cal A, \nu) \circlearrowleft$, it
is proved in \cite{rohklin} that $h_\nu(T)= h_{\nu^\sharp}(T^\sharp)$.



\small
\begin{thebibliography}{99999}

\bibitem[A]{abramov} L.M. Abramov, \emph{The entropy of a derived
    automorphism}, Dokl. Akad. Nauk. SSSR {\bf 128} (1959), 647-650.
    Amer. Math. Soc. Transl. {\bf 49}:2 (1966), 162-166.

\bibitem[B]{bowen}  R. Bowen, \emph{Equilibrium states and the ergodic theory
    of Anosov diffeomorphisms}.
    Lecture Notes in Math. {\bf 470}.  Springer-Verlag: Berlin, 1975.

\bibitem[BK]{brin katok}  M. Brin and A. Katok, \emph{On local entropy}, Geometric Dynamics (Rio de Janeiro, 1981), Lecture Notes in Math. {\bf 1007},
Springer-Verlag:  Berlin, 1983, p 30-38.

\bibitem[BDM]{bdm}  H. Bruin, M.F. Demers and I. Melbourne, \emph{Existence and convergence properties of physical measures for certain dynamical systems with holes},
Ergod. Th. and Dynam. Systems {\bf 30} (2010), 687-728.

\bibitem[Bu]{buzzi} J. Buzzi, \emph{Markov extensions for multidimensional dynamical
systems}, Israel J. of Math. {\bf 112} (1999), 357-380.

\bibitem[CM1]{chernov mark1}  N. Chernov and R. Markarian, \emph{Ergodic properties of Anosov maps with rectangular holes}, Bol. Soc. Bras. Mat. {\bf 28} (1997), 271-314.

\bibitem[CM2]{chernov mark2}  N. Chernov and R. Markarian, \emph{Anosov maps with rectangular holes.  Nonergodic cases}, Bol. Soc. Bras. Mat. {\bf 28} (1997), 315-342.

\bibitem[CM3]{chernov book}  N. Chernov and R. Markarian, \emph{Chaotic Billiards},
Mathematical Surveys and Monographs, {\bf 127}, AMS:  Providence, RI, 2006.

\bibitem[CMT]{chernov mt} N. Chernov, R. Markarian and S. Troubetzkoy,
\emph{Invariant measures for Anosov maps with small holes}, Ergod. Th. and Dynam.
Sys, {\bf 18} (1998), 1049-1073.

\bibitem[CMS]{collet} P. Collet, S. Mart\'{i}nez and B. Schmitt, \emph{The Yorke-Pianigiani measure and
    the asymptotic law on the limit Cantor set of expanding systems}, Nonlinearity {\bf 7} (1994), 1437-1443.

\bibitem[D]{demers norms}  M.F. Demers, \emph{Functional norms for Young towers},
Ergod. Th. and Dynam. Systems {\bf 30}:5 (2010), 1371-1398.

\bibitem[DWY]{dwy}  M.F. Demers, P. Wright and L.-S. Young, \emph{Escape rates and
physically relevant measures for billiards with small holes}, Commun. Math. Phys. {\bf 294}
(2010), 353-388.

\bibitem[DY]{demers young} M.F. Demers and L.-S. Young, \emph{Escape rates and conditionally invariant measures}, Nonlinearity {\bf 19} (2006), 377-397.

\bibitem[KS]{katok} A. Katok and J.-M. Strelcyn, with the collaboration of
F. Ledrappier and F. Przytycki, \emph{Invariant Manifolds,
Entropy and Billiards; Smooth Maps with Singularities}, Lecture Notes in Math.
{\bf 1222}, Springer-Verlag:  Berlin, 283 pages, 1986.

\bibitem[LY]{ledrappier young}  F. Ledrappier and L.-S. Young, \emph{The metric
entropy of diffeomorphisms.  Part I:  Characterization of measures satisfying Pesin's
entropy formula}, Annals of Math. {\bf 122} (1985), 509-539.

\bibitem[M]{mane}  R. Ma\~n\'e, \emph{A proof of Pesin's formula}, Ergod. Th. and
Dynam. Systems {\bf 1}:1  (1981), 95-102.

\bibitem[P]{pesin 1} Ya. B. Pesin, \emph{Characteristic Lyapunov exponents and
smooth ergodic theory}, Russian Math. Surveys {\bf 32}:4 (1977), 55-114.

\bibitem[Pe]{petersen} K.E. Petersen, \emph{Ergodic Theory}, Cambridge Studies in Advanced
Mathematics, Cambridge University Press: Cambridge, 329 pages, 1989.

\bibitem[Ro]{rohklin} V. Rohklin, \emph{Exact endomorphisms of Lebesgue space},
Amer. Math. Soc. Transl. (2) {\bf 39} (1964), 1-36.

\bibitem[R]{ruelle} D. Ruelle, \emph{An inequality for the entropy of differentiable maps}, Bol. Soc. Brasil. Mat. {\bf 9}:1 (1978), 83-87.

\bibitem[S]{sarig} O. Sarig, \emph{Thermodynamic formalism of countable
    Markov shifts}, Ergod. Th. and Dynam. Sys. {\bf 19} (1999), 1565-1593.

\bibitem[W]{walters} P. Walters, \emph{An Introduction to Ergodic Theory}, Graduate Texts in Mathematics {\bf 79}, Springer-Verlag: New York, 250 pages, 1982.

\bibitem[Y1]{young large d}  L.-S. Young, \emph{Some large deviation results for dynamical systems}, Trans. Amer. Math. Soc. {\bf 318}:2 (1990), 525-543.

\bibitem[Y2]{young notes} L.-S. Young, \emph{Ergodic theory of differentiable dynamical
systems}, ``Real and Complex Dynamical Systems," eds. B. Branner and P. Hjorth, NATO ASI Series C464,
Kluwer Academic Publishers (1995), 293-336.

\bibitem[Y3]{young tower}  L.-S. Young, \emph{Statistical properties of
    dynamical systems with some hyperbolicity}, Annals of Math.
    {\bf 147}:3 (1998), 585-650.

\bibitem[Y4]{young srb}  L.-S. Young, \emph{What are SRB measures and which dynamical systems have them?} J. Statist. Phys. {\bf 108} (2002), 733-754.

\end{thebibliography}
\end{document}